\documentclass[a4paper,12pt]{article}

\RequirePackage{amsthm,amsmath,amsfonts,amssymb}
\RequirePackage[numbers]{natbib}
\RequirePackage{graphicx}

\theoremstyle{plain}
\newtheorem{theorem}{Theorem}[section]
\newtheorem{proposition}[theorem]{Proposition}
\newtheorem{lemma}[theorem]{Lemma}
\newtheorem{corollary}[theorem]{Corollary}
\newtheorem{algorithm}[theorem]{Algorithm}

\theoremstyle{remark}
\newtheorem{example}[theorem]{Example}
\newtheorem{definition}[theorem]{Definition}
\newtheorem{remark}[theorem]{Remark}

\begin{document}

\begin{center}

  \Large
  {\bf Algorithm for Direct Sampling from
    Conditional Distributions of Toric Models}
  \normalsize

  \bigskip By \bigskip

  \textsc{Shuhei Mano} and \textsc{Nobuki Takayama}

  \smallskip
  
  (The Institute of Statistical Mathematics and
   Kobe University, Japan)


\end{center}

\small

{\bf Abstract.}
We show that contiguity relations of hypergeometric functions
of several variables give a direct sampling algorithm from
the conditional distribution of toric models in statistics.
The algorithm is based on a Markov chain on a lattice
generated by a matrix $A$. A correspondence between decomposable
graphical models and $A$-hypergeometric systems is discussed.
We give a sum formula of special values of $A$-hypergeometric
polynomials. Some examples with implementations are presented.

\smallskip

{\it Key Words and Phrases.}
$A$-hypergeometric system, discrete exponential family,
GKZ-hypergeometric system, graphical model, Markov chain
Monte Carlo

\smallskip

2020 {\it Mathematics Subject Classification Numbers.}
33C65, 33C90, 33F99, 62H17, 62R01, 65C05

\normalsize

\section{Introduction}
\label{sect:intr}

Consider a discrete sample with state space $[m]:=\{1,2,\ldots,m\}$
for $m\in\mathbb{N}$, where $\mathbb{N}$ is the set of positive
integers. In this paper, we will discuss the discrete exponential
families in statistics, which are also called log-affine models
or toric models.

\begin{definition}\label{defi:toric}
  Let $A=(a_{ij})$ $\in \mathbb{Z}^{d\times m}$ be a matrix of
  integers such that no row or column is the zero vector and
  $(1,\ldots,1)\in \text{rowspan}(A)$. Let
  $y\in\mathbb{R}^m_{>0}$. The {\it toric model} associated
  with $A$ is the set of probability distributions 
  \[
    \{p\in\text{int}(\Delta_{m-1}):
      \log p \in \log y+\text{rowspan}(A)\},
  \]
  where $\Delta_{m-1}$ is the standard $d$-dimensional simplex.
  If $y=1$, the model is called the {\it log-linear model.}  
\end{definition}

Suppose we have parameters $\phi\in\mathbb{R}_{>0}^d$. For the matrix
$A=(a_{ij})\in\mathbb{Z}^{d\times m}$ with
$(1,\ldots,1)\in \text{rowspan}(A)$, the vector $p$ in
Definition~\ref{defi:toric} may be parameterized as
\begin{equation}
  p_j=\frac{y_j}{Z(\phi,y)}
  \prod_{i\in[d]}\phi_i^{a_{ij}},
  \quad j \in [m]
  \label{para}
\end{equation}
with the normalization constant $Z(\phi,y)$ such that
$\sum_{j\in[m]} p_j=1$. The condition
$(1,\ldots,1)\in \text{rowspan}(A)$ is needed to admit 
the normalization constant. The set determined by the linear
transformation
\begin{equation}
  \mathcal{F}_A(b):=\{u:Au=b,u\in\mathbb{N}_0^m\},
  \quad \mathbb{N}_0:=\{0\}\cup\mathbb{N}
  \label{fiber}
\end{equation}
is called the $b$-{\it fiber} associated to
{\it configuration matrix} $A$ with the minimal sufficient
statistics
$b\in\mathbb{N}_0A:=\{\sum_{j\in[m]}\lambda_ja_j:\lambda_j\in\mathbb{N}_0\}$
for $\phi$, where $a_j$ denotes the $j$-th column vector of
$A$. We will use the notation $u_.:=\sum_{j\in[m]}u_j$, where
the centered dot ($\cdot$) in the position of the index $j$
means that the index $j$ is summed up. A sample consisting of
observations in which counts of the $j$-th state is $u_j$,
$j\in[m]$ taken from the multinomial distribution specified
by the probability mass function \eqref{para} follows
the probability law
\[
  \mathbf{P}(U=u,AU=b)=\frac{u_.!}{\{Z(\phi,y)\}^{u_.}}
  \prod_{i\in[d]}\phi_i^{b_i}
  \prod_{j\in[m]}\frac{y_j^{u_j}}{u_j!}.
\]
The conditional distribution given the minimal sufficient
statistics $b$ is
\begin{equation}
  \mathbf{P}(U=u|AU=b)=\frac{1}{Z_A(b;y)}\frac{y^u}{u!},
  \quad y^u:=\prod_{j\in [m]}y_j^{u_j}, \quad
  u!:=\prod_{j\in[m]}u_j!.
  \label{cond}
\end{equation}
The support is the $b$-fiber $\mathcal{F}_A(b)$.
Here, the normalization constant
\begin{equation}
  Z_A(b;y):=\sum_{u\in \mathcal{F}_A(b)}\frac{y^u}{u!}
  \label{Apol}
\end{equation}
is called the $A$-{\it hypergeometric polynomial,} or
the {\it GKZ-hypergeometric polynomial,} defined by Gel'fand,
Kapranov, and Zelevinsky in the end of the 1980's.
We adopt the convention $Z_A(b;y)=0$ if $b\notin\mathbb{N}_0A$.
It is straightforward to see that a sample consisting of
the Poisson random variables whose means are given by
the probability mass function \eqref{para} also follows
the conditional distribution \eqref{cond}.


A typical use of sampling from a conditional distribution
from a toric model in statistics is a hypothesis testing.
Testing the hypothesis $y=1$ (log-linear) is specifically
called Fisher's exact test. Consider a two-by-two contingency
table. Under the independence of rows and columns,
the conditional distribution given marginal sums is called
the hypergeometric distribution. It is known that sampling from
the hypergeometric distribution is achieved by an urn model
(see \cite{Sul18}, Chapter 9). The conditional
distribution with $y\neq 1$, where $y$ is the odds ratio of
the interaction between rows and columns, is called
the generalized hypergeometric distribution. To the best of
the authors' knowledge, in contrast to the hypergeometric
distribution, sampling from the generalized hypergeometric
distribution had not been known. Moreover, if a log-linear model
is not graphical, or graphical but not chordal (see
Subsection~\ref{subs:dec} for the definitions), sampling
from such models had not been known. It had been considered
that sampling from such models is difficult
(see \cite{Sul18}, Chapter~9). The difficulty was one of
the motivation that Hastings \cite{Has70} proposed the use
of the Metropolis chain, which is one of the most common
tool in the Markov chain Monte Carlo methods.
In the Metropolis chain, the unique stationary distribution
of an ergodic Markov chain is designed to be the distribution
which we need. Diaconis and Sturmfels \cite{DS98} proposed
the use of Gr\"obner bases of the toric ideal of a configuration
matrix $A$ to construct a basis of moves in the Metropolis
chain. They called such a basis, or a generating set for
the toric ideal, a {\it Markov basis.} The state space of
the chain is the $b$-fiber, and in an implementation of
the chain the ratio of two monomials, which correspond to
the current state and the proposed next state, should be
computed in each step of the chain. A comprehensive treatment
of Markov bases is \cite{AHT12}.

Recently, one of the authors \cite{Man17} showed that direct
sampling from the conditional distribution of any toric models
is possible. Extensive discussion on the algorithm is
available in the monograph \cite{Man18}.
Throughout this paper, we call sampling from a distribution
exactly {\it direct sampling} to distinguish it from
approximate samplings resorting the use of a Metropolis
chain. In the context of graphical models, it might be
surprising that direct sampling is possible even when
the Markov property represented by the interaction graph
does not hold, because the Markov property is the key for
direct sampling from decomposable graphical models ever known,
see, e.g., Section 4.4 of \cite{Lau94}. In fact, the proposed
algorithm in \cite{Man17} utilizes another kind of Markov
property which any toric model possesses.

The purpose of this paper is two fold. The first purpose is
to present an efficient implementation of the algorithm when
the transition probability of the Markov chain run by
the algorithm is not available, by introducing a notion of
the Markov lattice defined in Section~\ref{sect:lattice}.
The Markov lattice represents the Markov property mentioned
above. The second purpose is to present concrete examples
to show how the algorithm works with aids of 
contiguity relations of hypergeometric functions of several
variables for well-known toric models including graphical
models, because the paper \cite{Man17} focused on toric models
associated with configuration matrices of two rows.

This paper is organized as follows. In Section~\ref{sect:lattice},
after reviewing the direct sampling algorithm proposed in
\cite{Man17}, or Algorithm~\ref{algo:0}, the Markov lattice,
which is a bounded integer lattice generated by the configuration
matrix $A$, is introduced. A single run of Algorithm~\ref{algo:0}
produces a sample path of the Markov chain on the Markov lattice,
where the transition probability is given by the
$A$-hypergeometric polynomials. If a closed form of an
$A$-hypergeometric polynomial is known for a special value,
we will call it an $A$-{\it hypergeometric sum formula}.
For such a case, the implementation is straightforward.
Otherwise, how to compute the transition probability in each
step of the algorithm is the key for an efficient implementation.
It is shown that the computation can be done along with each
sample path on the Markov lattice. The method is summarized as
Algorithm~\ref{algo:1}. In Section~\ref{sect:sum}, the cases
that the $A$-hypergeometric sum formula is known are discussed.
The celebrated formula of the normalization constants of
decomposable graphical models given by Sundberg \cite{Sun75} is revisited as
an $A$-hypergeometric sum formula. In Section~\ref{sect:examp},
some concrete examples are discussed. Algorithms~\ref{algo:0}
and \ref{algo:1} are implemented for the univariate Poisson
regression model and two-way contingency tables, and
the performances are evaluated and compared with those of
approximate samplings by Metroplis chains. The no-$l(\ge 3)$-way
interaction model, which is an important
class of log-linear models, is also discussed.

\section{The direct sampling algorithm}
\label{sect:lattice}

In this section we introduce the {\it Markov lattice,} which
is a bounded integer lattice generated by the configuration
matrix. The notion of the Markov lattice aids to construct
the direct sampling algorithms.

Recall that the $A$-hypergeometric polynomial \eqref{Apol}
is a solution of the $A$-hypergeometric system.

\begin{definition}
  For a matrix $A=(a_{ij})\in \mathbb{Z}^{d\times m}$ and
  a vector $b\in\mathbb{C}^d$, the
  $A$-{\it hypergeometric system} is the following system of
  linear partial differential equations for an indeterminate
  function $f=f(y)$:
  \begin{align}
    &(\sum_{j\in[m]}a_{ij}\theta_j-b_i)\bullet f=0,
    \quad  i\in[d], \quad
    \theta_j:=y_j\partial_j, \quad
    \partial_j:=\frac{\partial}{\partial y_j}, \quad
    \text{and}
    \label{HA1}\\
    &(\partial^{z^+}-\partial^{z^-})\bullet f=0,\quad
    z\in\text{Ker} A\cap {\mathbb{Z}}^m, \quad
    \partial^{z}:=\prod_{j\in[m]}\partial_j^{z_j},
    \label{HA2}
  \end{align}
  where $z^+_j=\max\{z_j,0\}$, $z^-_j=-\min\{z_j,0\}$,
  and $\bullet$ denotes the action of a differential operator
  to a function.
  The second group of operators in \eqref{HA2} generates
  the {\it toric ideal} of $A$.
\end{definition}

If $(1,\ldots,1)\in\text{rowspan}(A)$, the condition
\eqref{HA1} demands homogeneity of the $A$-hypergeometric
polynomial $Z_A(b;y)$:
\[
  (\sum_{j\in[m]} \theta_j-\deg(b))
  \bullet Z_A(b;y)=0,
\]
where $\deg(b):=\deg(Z_A(b;y))$
and $\deg$ denotes the degree of a polynomial. For example,
$\deg(y^u):=|u|:=\sum_{j\in[m]}u_j$.
By applying $\partial_k$ to
\eqref{HA1} from the left yields
\[
  \{\sum_{j\in[m]}a_{ij}\theta_j-(b_i-\sum_{j\in[m]}a_{ij})\}
  \partial_k \bullet f=0,
\]
or
\begin{equation}
  \partial_j Z_A(b;y)= Z_A(b-a_j;y), \qquad j\in[m],
  \label{cont}
\end{equation}
which leads to
\[
  \sum_{j\in[m]}\frac{\mu(b,b-a_j;y)}{\deg(b)}=1,
  \quad
  \mu(b,b-a_j;y):=\frac{Z_A(b-a_j;y)}{Z_A(b;y)}y_j.
\]
Here, $\mu(b,b-a_j;y)$ is the expectation of $u_j$:
\[
  \mathbf{E}(U_j|AU=b)=\frac{1}{Z_A(b;y)}
  \sum_{v\in\mathcal{F}_A(b)} v_j\frac{y^v}{v!}
  =\frac{\partial_j Z_A(b;y)}{Z_A(b;y)}y_j=
  \frac{Z_A(b-a_j;y)}{Z_A(b;y)}y_j.
\]
Let us consider a Markov chain of a vector in $\mathbb{N}_0^d$
runs over an integer lattice embedded in $\mathbb{N}_0 A$ with
the transition probability
\begin{equation}\label{kernel}
  M(b,b-a_j):=\frac{\mu(b,b-a_j;y)}{\deg(b)},
  \quad b-a_j\in\mathbb{N}_0 A, \quad j\in[m],
\end{equation}
and $M(b,b-a_j)=0$ if $b-a_j\notin\mathbb{N}_0 A$.
Since the contiguity relation \eqref{cont} implies that
$\deg(b-a_j)=\deg(b)-1$, the unique absorbing state of the Markov
chain is the zero vector. A sample path of the Markov chain is
specified by a sequence $(j_1,\ldots,j_{\deg(b)})\in [m]^{\deg(b)}$,
and the probability is 
\begin{align*}
  \mathbf{P}\{(j_1,\ldots,j_{\deg(b)})&\}=
  M(b,b-a_{j_1})M(b-a_{j_1},b-a_{j_1}-a_{j_2})\\
  &\cdots M(b-a_{j_1}-\cdots-a_{j_{\deg(b)-1}},
    b-a_{j_1}-\cdots-a_{j_{\deg(b)}}=0)\\
  =&\frac{\mu(b,b-a_{j_1};y)}{\deg(b)}
    \frac{\mu(b-a_{j_1},b-a_{j_1}-a_{j_2};y)}{\deg(b)-1}\\
  &\cdots\frac{\mu(b-a_{j_1}-\cdots-a_{j_{\deg(b)-1}},0;y)}{1}
  =\frac{1}{\deg(b)!}\frac{y^u}{Z_A(b;y)},
\end{align*}
where $u_j:=|\{t:j_t=j,t\in[\deg(b)]\}|$, $j\in[m]$. Accounting
the order of appearance of $[m]$ in the multiset
$\{j_1,\ldots,j_{\deg(b)}\}$, we have
\[
  \mathbf{P}(U=u)=\frac{|u|!}{u_1!\cdots u_m!}
  \mathbf{P}\{(j_1,\ldots,j_{\deg(b)})\}
  =\frac{1}{Z_A(b;y)}\frac{y^u}{u!}.
\]
This argument is summarized as the following algorithm.

\begin{algorithm}[\cite{Man17}]\label{algo:0}~
  \begin{itemize}
  \item[] Input: A matrix $A\in\mathbb{Z}^{d\times m}$,
    vectors $y\in\mathbb{R}_{>0}^m$,  $b\in\mathbb{N}_0A$
    with total number of counts $n$.
  \item[] Output: A vector of counts $(u_1,\ldots,u_m)$
    following the conditional distribution \eqref{cond} given
    $b$.     
  \end{itemize}    
  \begin{itemize}
      \item [] Step 1: Set $\beta=b$.

    \smallskip
    
    {\bf For} $t=1,2,\ldots,n$ {\bf do}

  \item [] ~~ Step 2: Compute $M(\beta,\beta-a_j)$ in \eqref{kernel}
    for each $j\in[m]$.

  \item [] ~~ Step 3: Pick $j_t=j$ with probability
    $M(\beta,\beta-a_j)$, $j\in[m]$.

  \item [] ~~ Step 4: Set $\beta\leftarrow \beta-a_{j_t}$.

    \smallskip
    
    {\bf End for}

  \item [] Step 5: Output $u_j:=\#\{t\in[n]:j_t=j\}$, $j\in[m]$.
  \end{itemize}
\end{algorithm}

Algorithm~\ref{algo:0} simulates the Markov chain of a vector in
$\mathbb{N}^d_0$ discussed above. It runs over an integer lattice
embedded in $\mathbb{N}_0A$. To describe the Markov chain
precisely, let us introduce the notion of the {\it Markov lattice.}

\begin{definition}
  Consider a matrix $A=(a_{ij})\in\mathbb{Z}^{d\times m}$ and
  a vector
  $b\in\mathbb{N}_0A=\sum_{j\in[m]}\mathbb{N}_0a_j\subset\mathbb{N}^d_0$.
  The {\it Markov lattice} $\mathcal{L}_A(b)$ is the bounded
  integer lattice embedded in $\mathbb{N}_0A$ equipped with
  the partial order
  \[
    v \in \mathbb{N}_0A~~\text{and}~~v-a_j \in \mathbb{N}_0A 
    \quad \Rightarrow \quad v-a_j \prec v,
  \]  
  and the maximum and the minimum are $b$ and $0$, respectively.
\end{definition}

A lattice consists of a partially ordered set in which any two
elements have a unique supremum and a unique infimum. In fact,
the Markov lattice is a lattice, because the unique supremum and
infimum of two elements $v=\sum_{j\in[m]}v_ja_j$ and
$w=\sum_{j\in[m]}w_ja_j$ are $\sum_{j\in[m]}\max\{v_j,w_j\}a_j$
and $\sum_{j\in[m]}\min\{v_j,w_j\}a_j$, respectively.

The state space of the Markov chain run by Algorithm~\ref{algo:0}
is the elements in Markov lattice $\mathcal{L}_A(b)$, where
an element $v$ of the Markov lattice specifies the $v$-fiber
$\mathcal{F}_A(v)$, or the $A$-hypergeometric polynomial $Z_A(v;y)$.
The transitions with positive probabilities occur between
{\it neighbors}, or from an element $v$ to an element $u$ satisfying
$u\prec v$.

\begin{remark}
There is a remarkable correspondence between the direct sampling
discussed in this paper and the Metropolis chain discussed in
\cite{DS98}. In the Metropolis chain, transitions are along with
the Markov basis with probabilities given by the ratio of two
monomials within the $b$-fiber, or $\{y^v:v\in\mathcal{F}_A(b)\}$,
while in the direct sampling transitions are along with
the Markov lattice $\mathcal{L}_A(b)$ with probabilities given
by the polynomials whose support is $\mathcal{F}_A(v)$, $v\prec b$. 
\end{remark}

\begin{example}[Two-way contingency tables of the independence model]\label{exam:2x2:1}
  We discuss the $2\times 2$ contingency table, but
  the same argument holds for general two-way
  contingency tables. The $2\times 2$ table is
  \[
  \begin{array}{cc|c}
    u_{11}&u_{12}&u_{1\cdot}\\
    u_{21}&u_{22}&u_{2\cdot}\\
    \hline
    u_{\cdot1}&u_{\cdot2}&
  \end{array},
  \]
  where $u_{i\cdot}=u_{i1}+u_{i2}$ and $u_{\cdot j}=u_{1j}+u_{2j}$,
  $i,j\in[2]$. The configuration matrix, the vectors of
  the observed counts, and the vector of the sufficient
  statistics are
  \begin{equation*}
    A\!=\!\left(\begin{array}{cccc}
      1&1&0&0\\
      0&0&1&1\\
      1&0&1&0\\
      0&1&0&1
    \end{array}\right),\quad
    u\!=\!\left(
    \begin{array}{c}
      u_{11}\\u_{12}\\u_{21}\\u_{22}
    \end{array}
    \right),\quad {\rm and}\quad
    b=\left(
    \begin{array}{c}
      u_{1\cdot}\\u_{2\cdot}\\u_{\cdot1}\\u_{\cdot2}
    \end{array}
    \right),
  \end{equation*}
  respectively.
  The conditional distribution of the independence
  model is a log-linear model and \eqref{cond} becomes
  \[
    \mathbf{P}(U=u|AU=b)
    =\frac{1}{Z_A(b;1)}\frac{1}{u_{11}!u_{12}!u_{21}!u_{22}!}.
  \]
Figure~\ref{fig1} is the Hasse diagram of an example.
\begin{figure}
  \begin{center}
\includegraphics[width=0.5\textwidth]{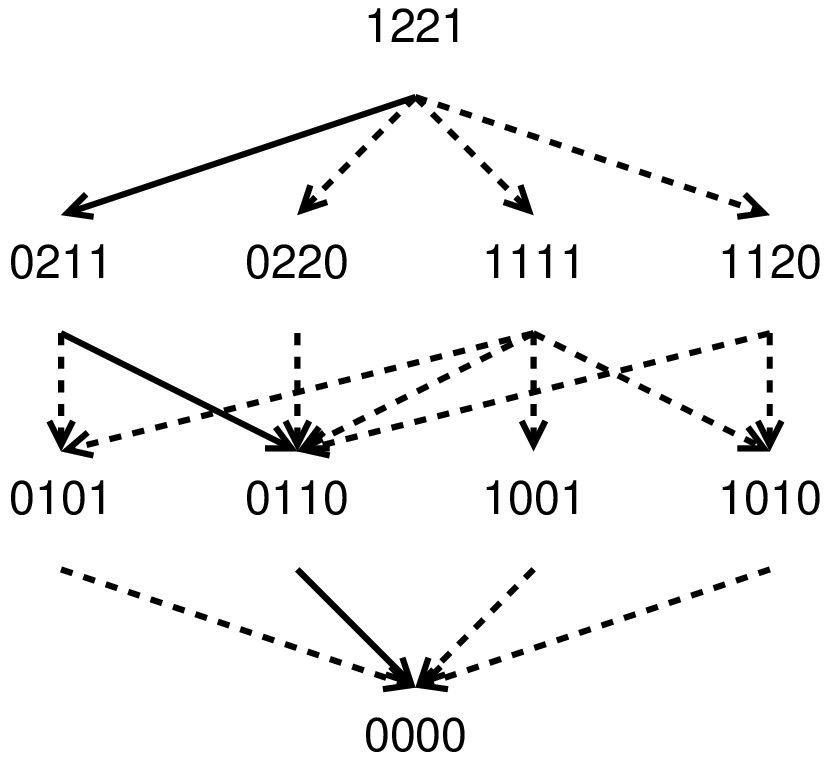}
\caption{The Markov lattice for the matrix $A$
  in Example~\ref{exam:2x2:1} with the maximum $b=(1,2,2,1)^\top$.
  A sample path is shown by solid edges.} 
\label{fig1}
\end{center}
\end{figure}
Transitions occur along the edges. The sample path,
shown with solid edges, is equal to one contingency
table with sufficient statistics $b=(1,2,2,1)^\top$.
In the first line of the following tables, asterisks
mark the cells picked during each solid edge; each
time a cell is picked, the corresponding marginal
count drops by one. The second line shows how
the contingency table is created by listing these
picked counts. 
\begin{align}
  \begin{array}{cc|c}
   *& &1\\
   & &2\\
  \hline  
  2& 1  
  \end{array}
  &~\quad\to&
  \begin{array}{cc|c}
   & &0\\
    &*&2\\
  \hline
  1& 1  
  \end{array}
  &~\quad\to&
  \begin{array}{cc|c}
   & &0\\
  *& &1\\
  \hline  
  1& 0  
  \end{array}
  &~\quad\to&
  \begin{array}{cc|c}
   & &0\\
    & &0\\
  \hline  
  0& 0  
  \end{array}\nonumber\\
  &
  \begin{array}{cc|c}
  1&0&1\\
  0&0&0\\
  \hline  
  1&0&   
  \end{array}& &
  \begin{array}{cc|c}
  1&0&1\\
  0&1&1\\
  \hline  
  1&1&   
  \end{array}& &
  \begin{array}{cc|c}
  1&0&1\\
  1&1&2\\
  \hline  
  2&1&   
  \end{array}&\label{tables}
\end{align}
\end{example}

To implement Step 2, we do not need to calculate the transition
probability $M(\beta,\beta-a_j)$ itself, because we can simply
calculate
\begin{equation}\label{mu_tilde}
  \tilde{\mu}(\beta,\beta-a_j;y):=Z_A(\beta-a_j;y)y_j, \quad
  j\in[m]
\end{equation}
and normalize them. If a closed form of
$\tilde{\mu}(\beta,\beta-a_j;y)$ is known, the computational
complexity of Algorithm~\ref{algo:0} is $O(dmn)$, whose dominant
contribution comes from updating $\beta$ for each $j\in[m]$.
Even otherwise, in principle, we can compute
$\tilde{\mu}(\beta,\beta-a_j;y)$ for any matrix $A$. We recall
how to compute the $A$-hypergeometric polynomial.

Let $H_A(b)$ denotes the left ideal of the Weyl algebra $D$
of indeterminants $y_1,\ldots,y_m$ generated by
the annihilators \eqref{HA1} and \eqref{HA2} with
the commutation rule
\[
  \partial_j h(y)=h(y)\partial_j+
  \partial_j \bullet h(y), \qquad h(y)\in\mathbb{Q}[y_1,\ldots,y_m].
\]
We call $H_A(b)$ the $A$-{\it hypergeometric ideal} with
parameters $b$.
Consider a coset of $D$ modulo $H_A(b)$, or $D/H_A(b)$.
A monomial $u\notin\text{in}_\prec(H_A(b))$
with respect to a monomial term order $\prec$ is called
a {\it standard monomial} with respect to the initial ideal
$\text{in}_\prec(H_A(b))$. The set of standard monomials
$\{\partial^{u(0)}=1,\partial^{u(1)},\ldots,\partial^{u(r-1)}\}$
forms a basis of $D/H_A(b)$ as
the vector space of the holomorphic solutions of the $A$-hypergeometric
system, where we call the dimension $r$ the {\it holonomic rank} of
$H_A(b)$. It is known that $r=\text{vol}(A)$, where
$\text{vol}(A)$ is the volume of the convex hull of
$\{0, a_1,\ldots,a_m\}$ with the unit being the volume of
the unit simplex $\{0, e_1,\ldots,e_d\}$ and $e_i$ is
a unit vector in $\mathbb{N}^d$ whose $i$-th component is 1
and the other components are zero \cite[Theorem 1.4]{GKZ90},
\cite[Theorem 3]{OT09}.
The hypergeometric system $H_A(b)$ can be transformed into
the {\it Pfaffian system}
\[
  \theta_j \bullet q(b;y)=P_j(b;y) q(b;y), \quad
  j\in [m]
\]
with
\begin{equation}
  q(b;y):=(\theta^{u(k)}Z_A(b;y):k\in 0\cup[r-1])^\top,
  \label{GMvec}
\end{equation}
where the elements of the $r\times r$ matrix $P_j(b;y)$
are rational functions of $y$, and the vector $q(b;y)$ will
be called the {\it Gauss--Manin vector} of $Z_A(b;y)$.
They are obtained via a Gr\"oebner basis normal form with
respect to a Gr\"obner basis of the hypergeometric ideal
generated by $H_A(b)$ \cite[Theorem 1.4.22]{SST00}.
By using the contiguity relation \eqref{cont}, we have
a recursion for the vector $q$,
\[
  y_jq(b-a_j;y)=\tilde{P}_j(b;y)q(b;y), \quad j\in[m],
  \quad q(0;y)=(1,0,\ldots,0)^\top,
\]
where
\[
  \tilde{P}_j(b;x)=(E_r+S_j)P_j(b;x), \quad
S_j=\left(\begin{array}{cccc}
  0 & 0 & \cdots & 0 \\
  -y_j^{-1}\theta^{u(1)}(y_j) & 0 & \cdots & 0 \\
  \vdots & \vdots & \ddots & \vdots \\
  -y_j^{-1}\theta^{u(r-1)}(y_j) & 0 & \cdots & 0
\end{array}\right)  
\]
and $E_r$ is the $r\times r$ identity matrix.
Ohara and Takayama \cite{OT15} showed that $q_0(b;y)=Z_A(b;y)$
is obtained by the matrix multiplications:
\begin{align}
  q(b;y)=&y_{j_1}\cdots y_{j_{\deg(b)}}\nonumber\\
  &\times
  \tilde{P}_{j_1}^{-1}(b;y)\tilde{P}^{-1}_{j_2}(b-a_{j_1};y)\cdots
  \tilde{P}^{-1}_{j_{\deg(b)}}(b-a_{j_1}-\cdots-a_{j_{\deg(b)-1}};y)
  \nonumber\\
  &\times q(0;y),
  \label{recQ}  
\end{align}
for a sequence of vectors $b$, $b-a_{j_1}$, $\ldots$, $0$,
where we assumed the matrices $\tilde{P}_j(b;y)$ are invertible
(see below for the assumption). They
called this method to obtain the Gauss--Manin vector
the difference holonomic gradient method. Hence, in principle,
we can obtain \eqref{mu_tilde}.
In practice, computation of a Gr\"obner basis is
expensive, but such computation can be avoided if the matrices
$P_j(\beta;y)$ are computed by another way with
using specific properties of each $A$-hypergeometric system.
See Section~\ref{sect:examp} for some examples.

Since $(E_r+S_j)$ are invertible, $\tilde{P}_j(b,y)$ are
invertible if and only if $P_j(b,y)$ are invertible.
If $D/H_A(b)$ and $D/H_A(b-a_j)$ are isomorphic,
$P_j(b,y)$ are invertible at a genetic point $y\in\mathbb{R}^m_{>0}$
and the resulting matrix multiplications are basis changes.
In this respect, we have the following proposition.

A matrix $A$ is said to be {\it normal} if
$\mathbb{N}_0A=\mathbb{Z}A\cap\mathbb{Q}_{\ge 0}A$, where
the cone $\mathbb{Q}_{\ge 0}A$ is defined as
$\mathbb{Q}_{\ge 0}A=\{\sum_{j\in [m]}\lambda_j a_j:\lambda_j\in\mathbb{Q}_{\ge 0}\}$.
Many important examples are nomal. In fact, all examples appear
in this paper are notmal, including the Aomoto--Gel'fand
hypergeometric systems in Subsection~\ref{subs:2}, the classical
univariate hypergeometric series \eqref{chyp}, the decomposable
graphical models in Subsection~\ref{subs:dec} \cite[Corollary 4.25] {HS02},
and the no-three-way interaction models of levels
$(r_1,r_2,r_3)=(2,2,2), (3,3,2)$ \cite[Theorem 2.2] {BHI11}.

For a facet $\sigma$ of $\mathbb{Q}_{\ge 0}A$, the linear form
$F_\sigma$ satisfying the following conditions is unique and
called the {\it primitive integral support function}:

\begin{itemize}

\item[(i)] $F_\sigma(\mathbb{Z}A)=\mathbb{Z}$;

\item[(ii)] $F_\sigma(a_j)\ge 0$ for all $j\in[m]$;

\item[(iii)] $F_\sigma(a_j)=0$ for all $a_j\in\sigma$.

\end{itemize}  
  
\begin{proposition}\label{prop:iso1}
  When a configuration matrix $A$ is normal, $D/H_A(\beta)$ and
  $D/H_A(\beta')$ are isomorphic as left $D$ modules for any
  $\beta, \beta' \in \mathcal{L}_A(b)$. In particular, there
  exists an isomorphism of the holomorphic solution spaces of
  them as vector spaces. 
\end{proposition}

\begin{proof}
It follows from \cite[Theorem 5.2] {Saito01} that
$M(\beta)$ and $M(\beta')$ are isomorphic if and only if
\[
\{ \sigma \ \text{facet} : F_\sigma(\beta) \in \mathbb{N}_0 \}
= 
\{ \sigma \ \text{facet} : F_\sigma(\beta') \in \mathbb{N}_0 \}.
\]
Since $\beta, \beta'\in \mathbb{N}_0 A$, we have
$F_\sigma(\beta), F_\sigma(\beta') \in \mathbb{N}_0$ for any facet
$\sigma$. Thus, the condition holds.
\end{proof}

An important non-normal example with the isomorphism is
the $A$-hypergeometric system associated with generic monomial curves,
or $d=2$, which was considered in depth in \cite{Man17}.
\cite[Theorem 6.3] {Saito01} gives the following proposition.
The proof is omitted here, as it is similar to Proposition~\ref{prop:iso1}.

\begin{proposition}
  When a configuration matrix $A$ of $d=2$, $D/H_A(\beta)$ and
  $D/H_A(\beta')$ are isomorphic as left $D$ modules for any
  $\beta, \beta' \in \mathcal{L}_A(b)$. In particular, there
  exists an isomorphism of the holomorphic solution spaces of
  them as vector spaces. 
\end{proposition}

We describe how a version of Algorithm~\ref{algo:0} is constructed.
Define a vector
\[
  \tilde{q}(v;y):=(\theta_1,\ldots,\theta_m)^\top\bullet Z_A(v;y),
  \quad v\in\mathcal{L}_A(b).
\]
Since the set of standard monomials forms a basis of the vector
space $R/H_A(b)$, there exists an $m\times r$ matrix $T(v,y)$
whose elements are rational functions of $y$ satisfying
\[
  \tilde{q}(v;y)=T(v;y)q(v;y),
\]
where $q(v;y)$ is the Gauss--Manin vector defined in \eqref{GMvec}.
The transition probabilities can be obtained along with a sample
path $(j_1,j_2,\ldots)$ as follows. For the first step, we compute
\[
  \tilde{\mu}(b,b-a_{j};y)=\tilde{q}_{j}(b;y)=
  \sum_{k\in 0\cup[r-1]}t_{jk}(b;y)q_k(b;y), \quad j\in[m].
\]
If $j=j_1$ is chosen, we update the Gauss--Manin vector
\begin{align*}
  q_k(b-a_{j_1};y)
  &=y_{j_1}^{-1}\{\tilde{P}_{j_1}(b;y)q(b;y)\}_k\\
  &=y_{j_1}^{-1}\sum_{l\in 0\cup[r-1]}\{\tilde{P}_{j_1}(b;y)\}_{kl}q_l(b;y),
  \quad k\in0\cup[r-1],
\end{align*}
which gives
\[
  \tilde{\mu}(b-a_{j_1};b-a_{j_1}-a_{j};y)=
  \sum_{k\in 0\cup[r-1]}t_{jk}(b-a_{j_1};y)q_k(b-a_{j_1};y),
  \quad j\in[m].
\]
Then, if $j=j_2$ is chosen, we update the Gauss--Manin vector
\[
  q_k(b-a_{j_1}-a_{j_2};y)
    =y_{j_2}^{-1}
    \sum_{l\in 0\cup[r-1]}\{\tilde{P}_{j_2}(b-a_{j_1};y)\}_{kl}
    q_l(b-a_{j_1};y), \quad k\in0\cup[r-1],
\]
which gives
\begin{align*}
  \tilde{\mu}(b-a_{j_1}-&a_{j_2};b-a_{j_1}-a_{j_2}-a_j)\\
  &=\sum_{k\in 0\cup[r-1]}t_{jk}(b-a_{j_1}-a_{j_2};y)
  q_k(b-a_{j_1}-a_{j_2};y),\quad j\in[m].
\end{align*}
Note that only the sequence of the Gauss--Manin vectors
$q_k(b;y)$, $q_k(b-a_{j_1};y)$, $q_k(b-a_{j_1}-a_{j_2};y)$,
\ldots along the sample path appears. This is a natural
consequence of Algorithm~\ref{algo:0} being a simulation
of the Markov chain, meaning that the transition probabilities
solely depend on the current state.

The above observation leads to the following version of
Algorithm~\ref{algo:0}.

\begin{algorithm}\label{algo:1}~
    \begin{itemize}
  \item[] Input: A matrix $A\in\mathbb{Z}^{d\times m}$,
    vectors $y\in\mathbb{R}_{>0}^m$,  $b\in\mathbb{N}_0A$
    with total number of counts $n$.
  \item[] Output: A vector of counts $(u_1,\ldots,u_m)$
    following the conditional distribution \eqref{cond} given
    $b$.     
  \end{itemize}    
  \begin{itemize}
    \item [] Step 1: Initialize the Gauss--Manin vector
      $q(b;y)$ defined in \eqref{GMvec} using the matrix
      multiplication \eqref{recQ}.
    \item [] Step 2: Set $\beta=b$.

    \smallskip
    
    {\bf For} $t=1,2,\ldots,n$ {\bf do}

  \item [] ~~ Step 3: Compute
    $\tilde{\mu}(\beta,\beta-a_j;y)=\{T(\beta;y)q(\beta;y)\}_j$
    for each $j\in[m]$.

  \item [] ~~ Step 4: Pick $j_t=j$ with probability
    proportional to $\tilde{\mu}(\beta,\beta-a_j;y)$, $j\in[m]$.

  \item [] ~~ Step 5: Compute
      \[
        q_k(\beta-a_{j_t};y)=
        y_{j_k}^{-1}
        \{\tilde{P}_{j_t}(\beta;y)q(\beta;y)\}_k, \quad k\in 0\cup[r-1],
      \]

  \item [] ~~ Step 6: Set $\beta\leftarrow \beta-a_{j_t}$.

    \smallskip
    
    {\bf End for}

  \item [] Step 5: Output $u_j:=\#\{t\in[n]:j_t=j\}$, $j\in[m]$.
  \end{itemize}

\end{algorithm}

Let us see the computational complexity of Algorithm~\ref{algo:1}.
We ignore the cost demanded by the initialization, or Step 1,
because we do not have to repeat the computation. Since Step 3
demands $O(mr)$, Step 5 demands $O(r^2)$, and the other steps
demand smaller costs, we have the following estimate.

\begin{proposition}\label{prop:cost}
  The computational complexity of Algorithm~\ref{algo:1} is
  $O(\max\{m,r\}rn)$, where $n=\deg(Z_A(b;y))$ and $r={\rm vol}(A)$,
  with assuming costs of rational arithmetic to be $O(1)$.
\end{proposition}

\section{Cases that sum formulae are known}
\label{sect:sum}

In this section we discuss Algorithm~\ref{algo:0} for the case
that the $A$-hypergeometric sum formula is known. After
warming up by considering 
  two-way contingency tables of the independence model
in Example~\ref{exam:2x2:1}, the celebrated
formula of the normalization constant of decomposable
graphical models given by Sundberg \cite{Sun75} is revisited.

For $\alpha,\beta,\gamma\in\mathbb{C}$, if
$\text{Re}(\gamma)>\text{Re}(\beta)>0$, the Gauss hypergeometric
series
\begin{equation}
  {}_2F_1(\alpha,\beta;\gamma;z)
  :=\sum_{u\ge 0}\frac{(\alpha)_u(\beta)_u}{(\gamma)_u}\frac{z^u}{u!},
  \quad z\in\mathbb{C},
  \label{Ghyp}
\end{equation}
where $(\alpha)_u$ is the Pochhammer symbol:
$(\alpha)_u:=\alpha(\alpha+1)\cdots(\alpha+u-1)$, has Euler's
integral
\[
  {}_2F_1(\alpha,\beta;\gamma;z)=\frac{\Gamma(\gamma)}{\Gamma(\beta)
    \Gamma(\gamma-\beta)}
  \int_0^1t^{\beta-1}(1-t)^{\gamma-\beta-1}(1-tz)^{-\alpha}dt.
\]
If $z=1$ the right hand side becomes a beta-integral and we have
the identity (Theorem 2.2.2 of \cite{AAR96})
\begin{equation*}
  {}_2F_1(\alpha,\beta;\gamma;1)
  =\frac{\Gamma(\gamma)\Gamma(\gamma-\alpha-\beta)}
  {\Gamma(\gamma-\alpha)\Gamma(\gamma-\beta)},
  \quad \text{Re}(\gamma)>\text{Re}(\alpha+\beta),
  \quad -\gamma\notin \mathbb{N}_0.
\end{equation*}
This sum formula of the Gauss hypergeometric series is called
the Gauss hypergeometric theorem. For the classical univariate
hypergeometric series
\begin{equation}
  {}_kF_{k-1}(c_1,\ldots,c_k;d_2,\ldots,d_k;z)
  :=\sum_{u\ge 0}\frac{(c_1)_u\cdots(c_k)_u}
  {(d_2)_u\cdots(d_k)_u}\frac{z^u}{u!}, \quad
  k\in\mathbb{N}\setminus\{1\},
  \label{chyp}
\end{equation}
several identities at $z=1$ are known, see, e.g. Sections 4.4
and 4.5 of \cite{Erd53}, Chapters 2 and 3 of \cite{AAR96}.
If the $A$-hypergeometric sum formula is available at
a specific value, say $y_0$, the computation of the transition
probability of the Markov chain run by Algorithm~\ref{algo:0}
with $y_0$ is straightforward. The independence model of
two-way contingency tables is such an example.

\subsection{Two-way contingency tables of the independence model}
\label{subs:2}

The two-way contingency table $(u_{ij}:i\in[r_1],j\in[r_2])$
given marginal sums is a toric model.
The conditional distribution given marginal sums $b$ is
given by the with matrix
\begin{equation}
  A=\left(\begin{array}{c}
  E_{r_1} \otimes 1_{r_2}\\
  1_{r_1} \otimes E_{r_2}
  \end{array}\right), \quad
  b=(u_{1\cdot},\ldots,u_{r_1\cdot},u_{\cdot1},\ldots,u_{\cdot r_2})^\top,
  \quad
  1_r:=(\overbrace{1,\ldots,1}^r),
  \label{Ab2}
\end{equation}
for the state vector
\[
  (u_{11},\ldots,u_{1r_2},u_{21},\ldots,u_{2r_2},\ldots,
   u_{r_11},\ldots,u_{r_1r_2})^\top,
\]
where $u_{1\cdot}=\sum_{j\in[r_2]}u_{1j}$ and $\otimes$
denotes the Kronecker product of matrices. 
The parameters are
\[
  z_{ij}=\frac{y_{ij}y_{r_1r_2}}{y_{ir_2}y_{r_1j}}, \quad
  z_{r_1i}=z_{jr_2}=z_{r_1r_2}=1, \quad i\in[r_1-1],~j\in[r_2-1].
\]
Here,
\[
  (u_{r_1\cdot}+u_{\cdot r_2}-n)!\prod_{i\in[r_1-1]}
  u_{i\cdot}!\prod_{j\in[r_2-1]}u_{\cdot j}!~Z_A(b;y)
\]
is a hypergeometric polynomial of type $(r_1,r_1+r_2)$,
or an Aomoto--Gel'fand hypergeometric polynomial associated with
the matrix $A$ and the vector $b$ in \eqref{Ab2}.
The hypergeometric polynomial of type $(r_1,r_1+r_2)$ is defined
as
\begin{align}
  &F(\alpha,\beta,\gamma;z)\nonumber\\
  &=\sum_{u\in\mathcal{F}_A(b)}
  \frac{
  \prod_{i\in[r_1-1]}(\alpha_i)_{u_{i\cdot}-u_{ir_2}}
  \prod_{j\in[r_2-1]}(\beta_j)_{u_{\cdot j}-u_{r_1j}}}
       {(\gamma)_{\sum_{i\in[r_1-1]}\sum_{j\in[r_2-1]}u_{ij}}}
  \prod_{i\in[r_1-1]}\prod_{j\in[r_2-1]}\frac{{z_{ij}}^{u_{ij}}}
  {u_{ij}!},
  \label{thyp}
\end{align}
where $\alpha=(-u_{1\cdot},\ldots,-u_{r_1-1,\cdot})$,
$\beta=(-u_{\cdot 1},\ldots,-u_{\cdot,r_2-1})$, and
$\gamma=u_{r_1\cdot}+u_{\cdot r_2}-u_{\cdot\cdot}+1$.
A hypergeometric function of type $(r_1,r_1+r_2)$ can be said
to be a hypergeometric function on a Grassmannian, and
various hypergeometric series appear as special cases. See
Section~3 of \cite{AK11} for the background.

If rows and columns are independent, or $y=1$, the model
reduces to the log-linear model discussed in Example~\ref{exam:2x2:1},
and the normalization gives a generalization of the Gauss
hypergeometric theorem:
\begin{equation}\label{GHT}
  Z_A(b;1)=\sum_{u\in\mathcal{F}_A(b)}\frac{1}{u!}
  =\frac{u_{\cdot\cdot}!}{\prod_{i\in[r_1]}u_{i\cdot}!
   \prod_{j\in[r_2]}u_{\cdot j}!},
\end{equation}
which can be confirmed with Euler's integral associated with 
the hypergeometric polynomial of type $(r_1,r_1+r_2)$
\cite[Theorem 3.3]{AK11}. Let $a_{(i,j)}$ denotes the column
vector of the matrix $A$ which specifies the $(i,j)$-entry of
the contingency table. The transition probability of the Markov
chain run by Algorithm~\ref{algo:0} becomes
\begin{equation}
  M(b,b-a_{(i,j)})=\frac{Z_A(b-a_{(i,j)};1)}{Z_A(b;1)}
  \frac{1}{u_{\cdot\cdot}}=\frac{u_{i\cdot}}{u_{\cdot\cdot}}
  \frac{u_{\cdot j}}{u_{\cdot\cdot}}.
  \label{urn}
\end{equation}

\subsection{Decomposable graphical models}
\label{subs:dec}

Let us recall some basic concepts around graphical models.
See Subsections 2.1 and 4.4 of \cite{Lau94} for the details.
Note that our discussion is restricted to discrete random
variables.

A {\it simplicial complex} with ground set $V$ is a set
$\Gamma\subseteq 2^V$ such that if $F\in\Gamma$ and
$F'\subseteq F$, then $F'\in\Gamma$. An element of $\Gamma$
will be called the {\it faces} of $\Gamma$, and the inclusion
maximal faces are the {\it facets} of $\Gamma$. A simplicial
complex is specified by listing its facets. For instance,
$\Gamma=[12][13][23]$ is the bracket notation of the simplicial
complex
\[
  \Gamma=\{\emptyset,\{1\},\{2\},\{3\},\{1,2\},\{1,3\},\{2,3\}\}.
\]
The brackets are also used to represent sets of integers, but
the distinction will be obvious from the context. Consider
the ground set $V=\{1,\ldots,|V|\}=[|V|]$ and
discrete random variables, $X_1,\ldots,X_{|V|}$, where $X_i\in[r_i]$
for some {\it level} $r_i\in\mathbb{N}$, $i\in V$.
If $r_1=\cdots=r_{|V|}=2$, the model is called {\it binary}.
The joint state space of a random vector $X=(X_1,\ldots,X_{|V|})$
is $\mathcal{I}_V=\prod_{i\in V}[r_i]$. For example, when
$\Gamma=[1][2]\subset 2^{\{1,2\}}$ and the levels are $r_1=2$
and $r_2=3$,
$\mathcal{I}_V=\{(1,1),(1,2),(1,3),(2,1),(2,2),(2,3)\}$.
For a state $i_V=(i_1,\ldots,i_{|V|})\in\mathcal{I}_V$ and
a subset $F=\{f_1,f_2,\ldots\}\subset V$, we write
$i_F=(\cdots,i_{f_1},\cdots,i_{f_2},\cdots,\ldots)$,
where the centered dots $(\cdot)$ are used to keep the places
of $V\setminus F$. In other words, $i_j$ of $i_F$ takes
a value in $[r_j]$ if $j\in F$ and $i_j=\cdot$ if $j\notin F$.
For example, for the binary simplicial complex
$\Gamma=[123][124]\subset 2^{\{1,2,3,4\}}$,
the set $\mathcal{I}_F$ with $F=\{1,2\}$ is
$\{(1,1,\cdot,\cdot),(1,2,\cdot,\cdot),(2,1,\cdot,\cdot),(2,2,\cdot,\cdot)\}$.
Commas in states will be omitted hereafter, if it causes
no confusion. The random vector $X_F=(X_f)_{f\in F}$ has
the state space $\mathcal{I}_F=\prod_{f\in F}[r_f]$.

\begin{definition}\label{defi:log-lin}
  Let $\Gamma\subseteq 2^V$ be a simplicial complex. For each
  facet $F\in\Gamma$, we introduce a set of $|\mathcal{I}_F|$
  positive parameters $\phi_{i_F}$.
  The {\it hierarchical log-linear model} associated with
  $\Gamma$ is the set of all probability distributions
  \begin{equation}\label{para2}
    \mathcal{M}_\Gamma=
    \left\{
      p\in\text{int}(\Delta_{|\mathcal{I}_V|-1}):
      p_{i_V}=\frac{1}{Z(\phi)}\prod_{F\in\text{facet}(\Gamma)}
      \phi_{i_F}, i_V \in \mathcal{I}_V
    \right\},
  \end{equation}
  where $Z(\phi)$ is the normalizing constant.
\end{definition}

Note that the parameterization \eqref{para2} is a form of
the parameterization \eqref{para} with $y=1$, if
the configuration matrix $A$ is appropriately chosen.
The configuration matrix $A$ is a binary matrix with
$m=|\mathcal{I}_V|$ columns and
$d=\sum_{F\in\Gamma}|\mathcal{I}_F|$ rows, and determines
the $b$-fiber \eqref{fiber}. Let $F,F'\subseteq V$. If states
of subsets $i_F\in\mathcal{I}_F$ and $i_{F'}\in\mathcal{I}_{F'}$
have no contradiction, that is, if all the elements are not
different between $i_F$ and $i_{F'}$ by regarding $\cdot$ as
a wildcard and the number of $\cdot$ in
$i_F$ is smaller than that of $i_{F'}$, we write
$i_F \subset i_{F'}$.
For example, $(11\cdot1) \subset (\cdot1\cdot1)$ but
$(11\cdot1) \not\subset (\cdot2\cdot1)$.
This is a consistent notation if we regard
the $i$-th $\cdot$ as the set $[r_i]$.
Let $u(i_V)\in\mathbb{N}_0^{\mathcal{I}_V}$ be an
$r_1\times\cdots\times r_{|V|}$ contingency table, which
is also denoted by $u_{i_V}$ for saving space. For any subset
$F=\{f_1,f_2,\ldots\}\subset V$, let $u(i_F)$ be
the $r_{f_1}\times r_{f_2}\times\cdots$ marginal table such
that $u(i_F)=\sum_{j\in\mathcal{I}_{V\setminus F}}u({i_F,j})$.
In other words, the symbol $u(i_F)$ is the sum of the count
$u(i_V)$ of the state $i_V\in\mathcal{I}_V$ such that
$i_V \subset i_F$. For example,
$u_{\cdot1\cdot2}=u(\cdot1\cdot2)=\sum_{i_1\in[r_1]}\sum_{i_3\in[r_3]} u(i_11i_32)$.
Hence, for a simplicial complex $\Gamma=[F_1][F_2]\cdots $,
the configuration matrix $A$ determines the linear
transformation $u(i_V)\mapsto (u(i_{F_1}),u(i_{F_2}),\ldots)$,
or minimal sufficient statistics for $\phi$.

For a hierarchical log-linear model $\mathcal{M}_{\Gamma}$
associated with the simplicial complex $\Gamma$, we associate
the undirected graph $\mathcal{G}(\Gamma)=(V,E)$ with edges
satisfying
\[
  a \sim b ~ \Leftrightarrow ~ \{a,b\} \subseteq
  F~\text{for~some}~F \in \Gamma.
\]
This graph is called the {\it interaction graph}. Different hierarchical
log-linear models may have the same interaction graph.
For instance, both simplicial complexes $[123]$ and $[12][13][23]$
have the complete three-graph as their interaction graph.
A hierarchical log-linear model is {\it graphical} if
the associated simplicial complex $\Gamma$ exactly consists
of the cliques (maximal complete subgraph) of its interaction
graph, namely, the facets of $\Gamma$ are the cliques of
$\mathcal{G}(\Gamma)$. Thus $[123]$ is graphical whereas
$[12][13][23]$ is not.

A subset $S$ of the vertex set $V$ is said to be
$(f,f')$-{\it separator} if all paths from a vertex $f$ to
$f'$ intersect $S$. The subset $S$ is said to separate $F$
from $F'$ if it is an $(f,f')$-separator for every $f\in F$,
$f'\in F'$. A triple $(F,F',S)$ of disjoint subsets of
the vertex set $V$ of an undirected graph $\mathcal{G}$ is
said to form a {\it decomposition} of $\mathcal{G}$ if
$V=F\cup F'\cup S$, where the {\it separator} $S$ is
a complete subset of $V$ and separates $F$ from $F'$.
We allow some of the sets in $(F,F',S)$ to be empty.
If the set of $F$ and $F'$ are both non-empty,
the decomposition is called {\it proper}. An undirected graph
is said to be {\it decomposable} if it is complete, or if
there exists a proper decomposition $(F,F',S)$ into
decomposable subgraphs $\mathcal{G}_{F\cup S}$ and
$\mathcal{G}_{F'\cup S}$. A graphical model is called
decomposable if the interaction graph is decomposable.
It is known that an undirected graph is decomposable if
and only if it is {\it chordal}, namely, every cycle of
length larger than three possesses a chord \cite[Proposition 2.5]{Lau94}.

Graphical models are used to study conditional independence
in the situation where we have a collection of random variables.
A probability measure of $X=(X_1,\ldots,X_{|V|})$ taking values
in the joint state space $\mathcal{I}_V=\prod_{i=1}^{|V|}[r_i]$
is said to be obeying the {\it global Markov property} relative
to an undirected graph $\mathcal{G}$, if for any triple
$(F,F',S)$ of disjoint subsets of $V$ such that $S$ separates
$F$ from $F'$ in $\mathcal{G}$, $X_F$ and $X_{F'}$ are
conditionally independent given $X_S$, which will be written as
\begin{equation}
  X_F \perp\!\!\!\perp X_{F'} | X_S. \label{markov}
\end{equation}
The global Markov property is represented by
the interaction graph $\mathcal{G}(\Gamma)$.

Sundberg \cite{Sun75} established the following $A$-hypergeometric
sum formula of parameter $y=1$ for decomposable graphical models
by using factorization of the probability measure of $X$
(see Subsection 4.4.1 of \cite{Lau94}).

On the other hand, it is known that the $A$-hypergeometric integral
\begin{equation}
  \int_\gamma
  \left(\sum_{j\in[m]}\prod_{i\in[d]} y_jt_i^{a_{ij}}\right)^n
  \prod_{i\in[d-1]}t_i^{-b_i-1} dt_i, \quad t_d=1,
  \label{GKZint}
\end{equation}
for a $(d-1)$-cycle $\gamma$ in
$\mathbb{C}^{*d}$ \cite[Theorem 2.7]{GKZ90}, \cite[Theorem 5.4.2]{SST00},
where $\mathbb{C}^*:=\mathbb{C}\setminus\{0\}$
and $n=\deg(Z_A(b;y))$, is annihilated by the $A$-hypergeometric
ideal $H_A(b)$. The integration at $y=1$ yields \eqref{GHTdec}.
In Appendix we will see such a derivation. See Appendix for
the definitions of the separator and the perfect sequence.

\begin{theorem}\label{theo:dec}
  Consider a decomposable graphical model associated with
  the simplicial complex $\Gamma$. The $A$-hypergeometric
  polynomial with parameters $y=1$ associated with the matrix
  $A_\Gamma$ in \eqref{Ahie} and the vector
  $b_\Gamma=u(i_C:i_C\in\mathcal{I}_C,C\in\mathcal{C})$ has
  an expression
  \begin{align}
    Z_{A_\Gamma}(b_\Gamma;1)&=
    \sum_{\{v(i_V):\sum_{i_V{\subset} i_C}v(i_V)
            =u(i_C),\forall C\in\mathcal{C}\}}
    \prod_{i_V\in\mathcal{I}_V}\frac{1}{v(i_V)!}\nonumber\\
    &=\frac{\prod_{S\in\mathcal{S}}
    \{\prod_{i_S\in\mathcal{I}_S}u(i_S)!\}^{\nu(S)}}
    {\prod_{C\in\mathcal{C}}\prod_{i_C\in\mathcal{I}_C}u(i_C)!},
    \label{GHTdec}
  \end{align}
  where $\mathcal{C}$ is the set of cliques of the interaction
  graph $\mathcal{G}(\Gamma)=(V,E)$, $\mathcal{S}$ is the set
  of separators with multiplicities $\nu$ in any perfect
  sequence, and $v(i_V)$ are variables, which takes values in
  $\mathbb{N}_0$, indexed by the elements of $\mathcal{I}_V$.
\end{theorem}

The transition probability of the Markov chain run by
Algorithm~\ref{algo:0} immediately follows by Theorem~\ref{theo:dec}.

\begin{corollary}
  Consider a decomposable graphical model associated with
  the simplicial complex $\Gamma$. The transition probability of
  the Markov chain run by Algorithm~\ref{algo:0} is given by
  \begin{equation}
    \mu(b_\Gamma,b_\Gamma-{a_{\Gamma}}_j;1)
    =\frac{\prod_{C\in\mathcal{C}}
    \prod_{\{i_C:i_C\in\mathcal{I}_C,i_C{\supset} i_j\}}u(i_C)}
    {\prod_{S\in\mathcal{S}}
    \{\prod_{\{i_S:i_S\in\mathcal{I}_S,i_S{\supset} i_j\}}
           {u(i_S)}\}^{\nu(S)}}, \quad \forall j\in[m],
    \label{trans}
  \end{equation}
  where ${a_{\Gamma}}_j$ is the $j$-th column vector of
  $A_\Gamma$ and $i_j$ is the index of the state specified
  by the $j$-th column vector.
\end{corollary}

\begin{example}\label{exam:4cyclet}
  Consider a binary decomposable graphical model with the
  simplicial complex $\Gamma=[123][124]$. The matrix
  is $A_{\Gamma}=\left(
    \begin{array}{c}
    E_8\otimes 1_2\\
    E_4\otimes 1_2\otimes E_2\\
    \end{array}
    \right)$ and the vector $b_{\Gamma}$ is
  \begin{align*}
    b_{\Gamma}=
    (&u_{111\cdot},u_{112\cdot},u_{121\cdot},u_{122\cdot},
      u_{211\cdot},u_{212\cdot},u_{221\cdot},u_{222\cdot},\\
     &u_{11\cdot1},u_{11\cdot2},u_{12\cdot1},u_{12\cdot2},
      u_{21\cdot1}, u_{21\cdot2}, u_{22\cdot1},u_{22\cdot2})^\top,
  \end{align*}
  where the columns of $A_\Gamma$ specify the joint states in
  \[
  \mathcal{I}_V=\{1111, 1112, 1121, 1122, 1211,\ldots, 2222\}.
  \]
  The separator is $\{1,2\}$. Consider sampling from
  the model by Algorithm~\ref{algo:0}. When we
    pick the cell specified by the first column vector, $j=1$,
  $(111\cdot),(11\cdot1),(11\cdot\cdot){\supset}(1111)=i_1$,
  and the transition probability of the Markov chain run
  by Algorithm~\ref{algo:0} given by \eqref{trans} becomes
  \[
  M(b_\Gamma,b_\Gamma-a_{\Gamma_1})=\frac{\mu(b_\Gamma,b_\Gamma-a_{\Gamma_1};1)}
       {u_{\cdot\cdot\cdot\cdot}}=
  \frac{u_{111\cdot}u_{11\cdot1}}{u_{11\cdot\cdot}
        u_{\cdot\cdot\cdot\cdot}}.
  \]
  On the other hand,
  since this decomposable graphical model has the global Markov property \eqref{markov} with $X_3 \perp\!\!\!\perp X_4|X_{\{1,2\}}$, we have
  \begin{align*}
    \mathbf{P}(X=(1,1,1,1))=&
    \mathbf{P}(X_3=1|X_1=X_2=1)\mathbf{P}(X_4=1|X_1=X_2=1)\\
    &\times\mathbf{P}(X_1=X_2=1)\\
    =&\frac{u_{111\cdot}}{u_{11\cdot\cdot}}
      \frac{u_{11\cdot1}}{u_{11\cdot\cdot}}
      \frac{u_{11\cdot\cdot}}{u_{\cdot\cdot\cdot\cdot}}
     =\frac{u_{111\cdot}u_{11\cdot1}}
      {u_{11\cdot\cdot}u_{\cdot\cdot\cdot\cdot}}.
  \end{align*}
  Therefore we may say that
  Algorithm~\ref{algo:0} provides a sampling utilizes this
  conditional independence.
\end{example}

\section{Examples}
\label{sect:examp}

This section is devoted to demonstrate the direct sampling from
toric models. The algorithm was implemented and the performance
for the univariate Poisson regression model and two-way contingency
tables are examined. The no-$l(\ge 3)$-way interaction model,
which is an important log-linear model, is also discussed.

Direct samplings and approximate samplings with the Metropolis chain
were implemented on a computer algebra system Risa/Asir version
\cite{NT92}. The following timing results were taken
on a CPU (Intel Core i5-4308U CPU, 2.80GHz) of a machine with
8GB memory. A single core was used unless otherwise noted.
Since we are interested in evaluation of performance
of algorithms without numerical errors, the following
computations were conducted in the rational number arithmetic.

A comparison of performance between a direct sampling and
an approximate sampling resorting the use of a Metropolis
chain is nontrivial, because steps of a Metropolis chain are
not independent. To account for the autocorrelation among
steps of a chain, we employed the notion of the effective
sample size. The effective sample size of $N$
steps is defined as $N/(1+2\sum^\infty_{t=1}\rho_t)$, where
$\rho_t$ is the autocorrelation of chi-squares at lag $t$.
The estimate was based on the sample autocorrelation,
and the sum was cut at the lag such that the sample
autocorrelation was larger than some small positive value,
where 0.05 was chosen.

The computational complexity of simulating
a Metropolis chain of random vectors of length $N$ is
$O(\max\{U,|\mathcal{B}|\}N)$, where $|\mathcal{B}|$ is
the size of a Markov basis $\mathcal{B}$ and $U$ is the number
of elements to be updated at each step of the chain.
The cost of $O(|\mathcal{B}|)$ comes from the choice of
an element from $|\mathcal{B}|$ elements.

\subsection{Univariate Poisson regressions}

Consider the univariate Poisson regression model, that is,
random variables $U_j$, $j\in[m]$ independently follow
the Poisson distribution with mean $\alpha+\beta j$, where
$\alpha$ and $\beta$ are the nuisance parameters.
The conditional distribution given the sufficient statistics
$b_1=\sum_{j=1}^m ju_j$ and $b_2=\sum_{j=1}^m u_j$ is
\eqref{cond} with
\begin{equation*}
A=\left(\begin{array}{cccc}
  1 & 2 & \cdots & m\\
  1 & 1 & \cdots & 1
\end{array}\right)
\label{A2row}
\end{equation*}
and $y=1$. An explicit expression of the Pfaffian system was
obtained by using a recurrence relation of the partial Bell
polynomials \cite{Man17}. Since $\text{vol}(A)=m-1$ and
$\deg(Z_A(b;y))=b_2$, if the $A$-hypergeometric sum formulae are
known, the computational complexity of the direct sampling
by Algorithm~\ref{algo:0} is $O(2mb_2)$. Otherwise, that by
Algorithm~\ref{algo:1} is $O(m^2b_2)$. A minimal Markov basis is
\[
  \mathcal{B}=\{e_i+e_j-e_{i+1}-e_{j-1}:1\le i<j\le m,i+2\le j\},
\]
where $|\mathcal{B}|=O(m^4)$. Proposition~\ref{prop:cost}
suggests that the computational complexities of the direct
sampling by Algorithms~\ref{algo:0} and \ref{algo:1} are
asymptotically smaller than that of a single step of
the Metropolis chain, if $m$ is sufficiently larger than
$\sqrt{b_1}$.

If $m\ge b_1-b_2+1$, the $A$-hypergeometric polynomial has
the closed form $(b_1-1)!/\{(b_2-1)!(b_1-b_2)!b_2!\}$, or
the $A$-hypergeometric sum formula, which is the signless Lah
number with divided by $b_1!$. Otherwise, no closed formula is
available. Nevertheless, the $A$-hypergeometric polynomials on
all the elements of the Markov lattice can be computed
efficiently by using the recurrence relation for the partial Bell
polynomials mentioned above. In fact, there exists bijection
between the partial Bell polynomials
$B_{b_1b_2}(\cdot!)=b_1!Z_A(b;1)$ and the elements
$(b_1,b_2)^\top$ of the Markov lattice, and computing the
partial Bell polynomials $B_{b_1b_2}(\cdot!)$ from the partial
Bell polynomials
$B_{0b_2}=\delta_{b_20}$ by using the recurrence relation
is equivalent to computing the hypergeometric polynomials of
all elements of the Markov lattice with the maximum
$(b_1,b_2)^\top$ from the minimum $0$. Therefore, we do not
need Algorithm~\ref{algo:1}.

As a benchmark problem, we considered the case of
$m=5$, $b=(288,120)^\top$, whose approximate sampling by the
Metropolis chain was discussed by Diaconis et al. \cite{DES98}.
Following \cite{DES98}, metropolis chains of 9,000 steps with
1,000 burn-in steps were generated. The average computation time
and the average effective sample size for a chain based on 100
trials were 10.72 seconds and 1,977.3, respectively. Computation of
the $A$-hypergeometric polynomials for all the elements of
the Markov lattice was 8.92 seconds. Then, we conducted
100 trials of generation of 1,977 tables by using
Algorithm~\ref{algo:0}. The average computation time was 5.50
seconds. The direct sampling by Algorithm~\ref{algo:0}
is more efficient than the Metropolis chain. We have little
reason to use the Metropolis chain for the problem.

\subsection{Two-way contingency tables}

Two-way contingency tables were discussed in Subsection~\ref{subs:2}.
In this subsection, we consider models with and without independence
of rows and columns. For an $A$-hypergeometric system associated
with two-way contingency tables, an explicit expression for
the Pfaffian system was obtained via computation of intersection
forms of the twisted cohomology groups in associated hypergeometric
integral \cite{GM18}. For an $r_1\times r_2$-table,
we have $m=r_1r_2$,
\[
  r=\text{vol}(A)
  =\left(\begin{array}{c}r_1+r_2-2\\r_1-1\end{array}\right)
\]
(see Subsection 3.6 of \cite{AK11}), and
$n=\deg(Z_A(b;y))=u_{\cdot\cdot}$. Goto and Matsumoto \cite{GM18}
discussed evaluation of the Gauss--Manin vector using the matrix
multiplication \eqref{recQ}, and the procedure was summarized as
Algorithm~7.8 in \cite{GM18}. As for Algorithm~\ref{algo:1},
Step 3 is their Corollary~7.1, and the update of the Gauss--Manin
vector in Step 6 is their Corollary~6.3. See Tachibana et al.
\cite{TGKT20} for explicit expressions for the binary two-way
contingency table. Moreover, Tachibana et al. discussed
an efficient implementation of the matrix multiplication
\eqref{recQ} by the modular method in computational algebra.
According to Proposition~\ref{prop:cost}, the computational complexity
of the direct sampling by Algorithm~\ref{algo:1} is
$O(\max\{r_1r_2,r\}rn)$,
where $r^2n$ dominates for large $r$, because $r$ grows
rapidly than $r_1r_2$. For example,
$(r_1r_2,r)=(4,2),(9,6),(16,20),(25,70)$ for $r_1=r_2=2,3,4,5$,
respectively. The unique minimal Markov basis up to sign is
$\mathcal{B}=(z_{ij}:i\in[r_1],j\in[r_2])$, where
\begin{equation}
  z_{ij}=\left\{
  \begin{array}{rl}
  +1& ~(i,j)=(i_1,j_1),(i_2,j_2),\\
  -1& ~(i,j)=(i_1,j_2),(i_2,j_1),\\
   0& ~(i,j)=\text{others},
  \end{array}
  \right.\qquad 1\le i_1<i_2\le r_1, \,\, 1\le j_1<j_2\le r_2,
  \label{MB2}
\end{equation}
and $|\mathcal{B}|=O(r_1^2r_2^2)$. Roughly speaking, if
$r_1$ and $r_2$ are fixed, the ratio of computational
complexities of the direct sampling by Algorithm~\ref{algo:1}
to that of a single step of the Metropolis chain scales with
$n$.

Direct samplings of two-way contingency tables by
Algorithms~\ref{algo:0} and \ref{algo:1}, and by
the Metropolis chain were implemented. The implementation
of Algorithm~\ref{algo:1} is published as the package
{\tt gtt\_ds} for Risa/Asir. As a benchmark problem, we
considered sampling $3\times 4$ tables whose marginal sums
given by
\[
  \begin{array}{cccc|c}
    & &  &  &10\\
    & &  &  &14\\
    & &  &  &26\\
    \hline
    6&9&15&20&
  \end{array}.
\]
For the independence model, 100 Metropolis chains of 10,000
steps with 1,000 burn-in steps were generated. We do not need
Algorithm~\ref{algo:1}, because we know the transition
probability \eqref{urn}. The average
computation time and the average effective sample size for
a chain were 1.80 seconds and 425.1, respectively. Then, we
conducted 100 trials of generation of 425 tables by using
Algorithm~\ref{algo:0}. The average computation time was
0.21 seconds. We may say that the direct sampling by
Algorithm~\ref{algo:0} is nine times efficient than
the Metropolis chain.

As a model without independence, we
considered parameters
\[
  (z_{ij})=\left(
  \begin{array}{cccc}
    1/2&1/11&1/13&1\\
    1/7&1/3 &1/5 &1\\
    1  &1   &1   &1\\
  \end{array}
  \right).
\]
We generated 100 Metropolis chains of 10,000 steps with 1,000
burn-in steps. The average computation time and the average
effective sample size for a chain were 1.96 seconds and 280.9,
respectively. Computing the $A$-hypergeometric polynomials
for all the elements of the Markov lattice is practically
impossible, because the number of the elements is huge.
We conducted 10 trials of generation of 281 tables by using
Algorithm~\ref{algo:1}. The average computation time was
1645.6 seconds. The same computation was performed in 693.6
seconds by using four cores. Currently, our implementation of 
Algorithm~\ref{algo:1} is far inefficient than the Metropolis
chain. However, the fact that tables generated by
Algorithm~\ref{algo:1} follow independently to the exact
distribution is a remarkable advantage, since the computation
can be performed in parallel. The timing is
practically reasonable if we use multiple cores of a CPU.

\begin{remark}\label{rema:mixing}
  Propp and Wilson proposed the coupling from the past (CFTP)
  algorithm, which guarantees that a sample taken from
  an ergodic Markov chain is a sample taken from the unique
  stationary distribution \cite{PW96}. The CFTP is another
  algorithm for taking samples exactly following the distribution
  which we need. The waiting time for taking the sample,
  called the coalescence, determines the performance. 
  The mixing time of an ergodic Markov chain is defined as
  \[
    t_{\text{mix}}(\epsilon):=\min\{t:
    \max_{x\in\mathcal{X}}\|M^t(x,\cdot)-\pi\|_{\text{TV}}<\epsilon\}
  \]
  for some $\epsilon>0$, where $\|\cdot \|_{\text{TV}}$ is
  the total variation distance, $M$ is the transition matrix,
  and $\pi$ is the stationary distribution. For sampling
  two-row ($r_1=2$) contingency tables from the uniform
  distribution, not the conditional distribution given
  marginal sums, Dyer and Greenhill \cite{DG00} showed that
  the upper bound of the mixing time of a Markov chain with
  moves of multiples of the Markov basis \eqref{MB2} at each
  step is $O(r_2^2\log n)$. By using the result, Kijima and
  Matsui \cite{KM06} constructed a CFTP algorithm
  \cite{PW96}, whose expected waiting time of the coalescence
  is $O(r_2^3\log n)$. For two-row contingency tables from
  the uniform distribution, their CFTP has smaller computational
  complexity than Algorithm~\ref{algo:1}. However, note that
  the path coupling and the estimate of the waiting times
  rely on specific properties of two-row contingency
  tables and the uniform distribution.
\end{remark}

\subsection{No-$l(\ge 3)$-way interaction models}

For an $l$-way contingency table
$(u_{i_1\ldots i_l}:i_1\in[r_1],\ldots,i_l\in[r_l])$,
the no-$l$-way interaction model is the hierarchical
log-linear model associated with the simplicial complex
$[1\cdots l]\setminus\{1,\ldots,l\}$. For example,
$[12]\setminus\{1,2\}$ is the simplicial complex $[1][2]$,
which is the independence model of two-way contingency
tables discussed in Subsection~\ref{subs:2}, and
$[123]\setminus\{1,2,3\}$ is the simplicial complex
$[12][23][31]$. If $l\ge 3$, the model is not graphical.
The classical univariate hypergeometric
polynomial \eqref{chyp} appears as the $A$-hypergeometric
polynomial in the binary no-$l$-way interaction model. 
For example, the hypergeometric polynomial $_4F_3$
appears in the binary no-3-way interaction model of
$2\times2\times2$ contingency tables, and the hypergeometric
polynomial $_8F_7$ appears in the binary no-4-way
interaction model of $2\times2\times2\times2$ contingency
tables.

Let us recall the following useful fact. The classical univariate
hypergeometric series \eqref{chyp} can be represented as
the $A$-hypergeometric series (Example 5.4.6 of \cite{SST00}),
while the classical univariate hypergeometric series satisfies
the following ordinary differential equation, see, e.g.,
Section 4.2 of \cite{Erd53}:
\begin{align}
  \left\{\prod_{i\in[k]}(\theta+d_i-1)-z\prod_{i\in[k]}(\theta+c_i)
  \right\}\bullet {}_kF_{k-1}(c_1,\ldots,c_k;d_2,\ldots,d_k;&z)=0,
  \nonumber\\
  k\ge 2,&\label{cl_rec0}
\end{align}
where $d_1=1$ and $\theta=z\partial_z$. This ordinary differential
equation yields the following recurrence relation:
\begin{align}
  \theta^k=&\frac{\sum_{i\in[k]}d_i-k-z\sum_{i\in[k]}c_i}{z-1}
  \theta^{k-1}\nonumber\\
  &+
  \frac{\sum_{2\le i<j\le k}(d_i-1)(d_j-1)-z\sum_{1\le i<j\le k}
  c_ic_j}{z-1}\theta^{k-2}\nonumber\\
  &+\cdots
  +\frac{\prod_{2\le i\le k}(d_i-1)-z\sum_{i\in[k]}\prod_{j\neq i}c_j}
  {z-1}\theta-\frac{z\prod_{i\in[k]}c_i}{z-1}
  \label{cl_rec1}
\end{align} 
on ${}_kF_{k-1}(c;d;z)$, where $z\neq 1$. For $z=1$, we have
\begin{align}
  \theta^{k}=&
  \frac{\sum_{2\le i<j\le k}(d_i-1)(d_j-1)-\sum_{1\le i<j\le k}c_ic_j-\sum_{i\in[k]}c_i}{k+1+\sum_{i\in[k]}(c_i-d_i)}\theta^{k-1}+\cdots\nonumber\\
  &-\frac{\sum_{i\in[k]}\prod_{j\neq i}c_j+\prod_{i\in[k]}c_i}
  {k+1+\sum_{i\in[k]}(c_i-d_i)}\theta
  -\frac{\prod_{i\in[k]}c_i}{k+1+\sum_{i\in[k]}(c_i-d_i)},
  \label{cl_rec2}
\end{align} 
on ${}_kF_{k-1}(c;d;1)$, where $\sum_{i\in[k]}(d_i-c_i)\neq k+1$.
Here, \eqref{cl_rec2} is obtained by applying $\theta$ to
\eqref{cl_rec0} from the left, and $\theta^i \bullet {}_kF_{k-1}(c;d;1)$
should be read as
\[
  \theta^i \bullet {}_kF_{k-1}(c;d;1)=
  \left\{\left(z\frac{\partial}{\partial z}\right)^i{}_kF_{k-1}(c;d;z)
  \right\}_{z=1}.
\]
These recurrence relations are useful to obtain Pfaffian systems
for univariate $A$-hypergeometric polynomials.
In addition, the initial Gauss--Manin vector of an univariate
$A$-hypergeometric polynomial can be evaluated efficiently by
the binary splitting algorithm, see, e.g., \cite{Mez10}.

A joint state of a binary $l$-way contingency table is
$i_V=(i_1\ldots i_l)$, $i_1\in[2],\ldots,i_l\in[2]$ with the joint
state space $\mathcal{I}_V=[2]^l$. A contingency table is
the set of $u(i_V)=u(i_1\ldots i_l)=u_{i_1\ldots i_l}$, $i_V\in[2]^l$.
  
\begin{proposition}\label{prop:no-l-way}
  The normalization constant of the conditional distribution
  of the binary no-$l(\ge 2)$-way interaction model is
  proportional to the classical univariate hypergeometric
  polynomial \eqref{chyp} with $k=2^{l-1}$.
\end{proposition}  

\begin{proof}
  The joint state space $\mathcal{I}_V=[2]^l$ can be identified
  with the hypercube and the sufficient statistics are the facets.
  For each state $i_V\in[2]^l$, there exists a shortest path to
  the state $(1\cdots 1)$ whose length is the number of coordinates
  of $i_V$ occupied by 2. Then, $u(i_V)$ can be expressed by
  an alternating sum of sufficient statistics and $u_{1\cdots 1}$.
  For example, the state $(1212)$ has a path
  $(1212)\to(1112)\to(1111)$, which can expressed as
  \[
    u_{1212}=u_{1\cdot12}-u_{1112}=u_{1\cdot12}-u_{111\cdot}
    +u_{1111}.
  \]
  For such an expression, $u(i_V)!$ can be represented by
  the Pochhammer symbol with subscript $u_{1\cdots 1}$, and
  $y_{i}^{u}$ can be represented by power of $y_{1\cdots1}$.
  For $u_{1212}!$,
  \[
    u_{1212}!
    =(u_{1\cdot12}-u_{111\cdot})!
    (u_{1\cdot12}-u_{111\cdot}+1)_{u_{1111}}
  \]
  and $y_{1212}^{u_{1212}}\propto y_{1212}^{u_{1111}}$.
  For the state $(1112)$, the unique shortest path to $(1111)$ gives
  \[
    u_{1112}!=(u_{111\cdot}-u_{1111})!=\frac{u_{111\cdot}!}
    {(-u_{111\cdot})_{u_{1111}}}(-1)^{u_{1111}}
  \]
  and $y_{1222}^{u_{1222}}\propto y_{1222}^{-u_{1111}}$.
  These observations lead to 
  \begin{equation}
    \frac{y^u}{u!}\propto
    \frac{\prod_{\{i:i\in[2]^l,~\text{number~of}~1~\text{is~odd}\}}
    (c_i)_{u_{1\cdots1}}}
    {\prod_{\{i:i\in[2]^l,~\text{number~of}~1~\text{is~even}\}}
    (d_i)_{u_{1\cdots1}}}
    \left(
    \frac{
    \prod_{\{i:i\in[2]^l,~\text{number~of}~1~\text{is~even}\}}y_i}
    {\prod_{\{i:i\in[2]^l,~\text{number~of}~1~\text{is~odd}\}}y_i}
    \right)^{u_{1\cdots1}},
    \label{toric2n}     
  \end{equation}
  which is the form of the summand of the classical univariate
  hypergeometric polynomial \eqref{chyp} with $k=2^{l-1}$ and
  $u=u_{1\cdots1}$. 
\end{proof}

Let $A^{(2^l)}$ denotes the configuration matrix. The above proof
gives some insights of the toric ideal and the $A$-hypergeometric
ideal generated by $A^{(2^l)}$. The expression \eqref{toric2n}
shows that the toric ideal is the principal ideal generated by
the binomial
\[
  \prod_{\{i:i\in[2]^l,~\text{number~of}~1~\text{is~even}\}}\partial_i
 -\prod_{\{i:i\in[2]^l,~\text{number~of}~1~\text{is~odd}\}}\partial_i,
\]
which is the unique minimal Markov basis up to sign.
The recurrence relation \eqref{cl_rec1} shows that
the holonomic rank of $A$-hypergeometric ideal is $2^{l-1}$
and $n={\rm deg}(Z_{A^{(2^l)}}(b;y))=u_{\ldots}$.
According to Proposition~\ref{prop:cost}, computational complexity
of the direct sampling by Algorithm~\ref{algo:1} is $O(2^{2l-1}n)$,
while that of the Metropolis chain of length $N$ is $O(2^lN)$.
By using the recurrence relation \eqref{cl_rec1}, the explicit
form of the Pfaffian system can be obtained immediately.
The case of $l=3$ is as follows.

\begin{example}\label{exam:3way}  
  Consider the binary no-three-way interaction model associated
  with the simplicial complex $[12][13][23]$. Let
  \[
    A^{(2^3)}=
    \left(
    \begin{array}{c}
    E_{2^2} \otimes 1_2\\
    E_2 \otimes 1_2 \otimes E_2\\ 
    1_2 \otimes E_{2^2}
    \end{array}
    \right), ~
    b^\top=(
    u_{11\cdot},u_{12\cdot},u_{21\cdot},u_{22\cdot},
    u_{1\cdot1},\ldots,u_{2\cdot2},u_{\cdot11},\ldots,u_{\cdot22}),
  \]
  where the columns of $A^{(2^3)}$ specify the joint states in
  \[
  \mathcal{I}_V=\{111,112,121,122,211,\ldots,222\}.
  \]
  The conditional distribution of $u=u_{111}$ given marginal sums is
  \begin{align*}
    \mathbf{P}(U=u|A^{(2^3)} u=b)
    =\frac{1}{Z_{A^{(2^3)}}(b;y)}\frac{y^u}{u!}
    =\frac{1}{{}_4F_3(c;d;z)}
    \frac{(c_1)_u(c_2)_u(c_3)_u(c_4)_u}
    {(d_2)_u(d_3)_u(d_4)_u}\frac{z^u}{u!},
  \end{align*}
  where
  \begin{align*}
  &c=(-u_{11\cdot},-u_{\cdot11},-u_{1\cdot1},
      -u_{\cdot\cdot\cdot}+u_{1\cdot\cdot}+u_{\cdot1\cdot}
      +u_{\cdot\cdot1}-u_{11\cdot}-u_{\cdot11}-u_{1\cdot1}),\\
  &d=(1,u_{1\cdot\cdot}-u_{11\cdot}-u_{1\cdot1}+1,
      u_{\cdot1\cdot}-u_{11\cdot}-u_{\cdot11}+1,
      u_{\cdot\cdot1}-u_{1\cdot1}-u_{\cdot11}+1),
  \end{align*}
  and
  \[
  z=\frac{y_{111}y_{122}y_{212}y_{221}}{y_{112}y_{121}y_{211}y_{222}}.
  \]
  Therefore,
  \[
    Z_{A^{(2^3)}}(b;y)=cy^s{}_4F_3(c_1,c_2,c_3,c_4;d_2,d_3,d_4;z),
  \]
  where  
  \[
  c=\frac{1}{(-c_1)!(-c_2)!(-c_3)!(-c_4)!d_2!d_3!d_4!}, \quad
  y^s=\frac{y_{122}^{d_2-1}y_{212}^{d_3-1}y_{221}^{d_4-1}}
    {y_{112}^{c_1}y_{121}^{c_2}y_{211}^{c_3}y_{222}^{c_4}}.
  \]
  To avoid messy expressions with many subscripts, let
  \[
  (y_1,y_2,y_3,y_4,y_5,y_6,y_7,y_8)=
  (y_{111},y_{112},y_{122},y_{121},y_{212},y_{211},y_{221},y_{222}).
  \]
  The toric ideal of the matrix $A^{(2^3)}$ is the principal ideal
  generated by the binomial 
  $\partial_1\partial_3\partial_5\partial_7-\partial_2\partial_4\partial_6\partial_8$,
  and the set of the standard monomials in reverse lexicographic
  term order with $\partial_1\succ\partial_2\succ\cdots$ is
  $\{1,\theta_8,\theta_8^2,\theta_8^3\}$, where
  $\theta_1=\theta_3=\theta_5=\theta_7=\theta$ and
  $\theta_2=\theta_4=\theta_6=\theta_8=-\theta$.
  The Gauss--Manin vector is
  $(1,\theta_8,\theta^2_8,\theta^3_8)Z_{A^{(2^3)}}(b;y)$.
  The Pfaffian system is given in Appendix, where we assume $u_{\cdot\cdot\cdot}\ge 2$.
\end{example}

If a level is larger than two, the associated $A$-hypergeometric
polynomial becomes much more complicated. Consider the
no-three-way interaction model with $r_1=r_2=3$ and $r_3=2$. Let
the matrix
\[
  A^{(332)}=\left(
  \begin{array}{c}
    E_{3^2}\otimes 1_2\\
    E_3\otimes 1_3\otimes E_2\\
    1_3\otimes E_{3\times2}
  \end{array}\right), \quad
  b^{(332)\top}=(
  u_{11\cdot},\ldots,u_{33\cdot},u_{1\cdot1},\ldots,u_{3\cdot2},
  u_{\cdot11},\ldots,u_{\cdot32}),
\]
where the columns of $A^{(332)}$ specify the joint states in
\[
\mathcal{I}_V=\{111, 112, 121, 122, 131, 132, 211, \ldots, 332\}.
\]
The model can be represented by the following three tables:
\begin{equation*}
  \begin{array}{ccc|c}
  u_{111}    &u_{121}    &u_{131}    &u_{1\cdot1}\\
  u_{211}    &u_{221}    &u_{231}    &u_{2\cdot1}\\
  u_{311}    &u_{321}    &u_{331}    &u_{3\cdot1}\\
  \hline
  u_{\cdot11}&u_{\cdot21}&u_{\cdot31}&\\
  \end{array},\quad
  \begin{array}{ccc|c}
  u_{112}    &u_{122}    &u_{132}    &u_{1\cdot2}\\
  u_{212}    &u_{222}    &u_{232}    &u_{2\cdot2}\\
  u_{312}    &u_{322}    &u_{332}    &u_{3\cdot2}\\
  \hline
  u_{\cdot12}&u_{\cdot22}&u_{\cdot32}&\\
  \end{array},\quad
  \begin{array}{ccc|c}
  u_{11\cdot}    &u_{12\cdot}    &u_{13\cdot}    &~~~\\
  u_{21\cdot}    &u_{22\cdot}    &u_{23\cdot}    &~~~\\
  u_{31\cdot}    &u_{32\cdot}    &u_{33\cdot}    &~~~\\
  \hline
  &&&\\
  \end{array}.
\end{equation*}
Since the dimension of the kernel of $A^{(332)}$ is four, we take
$(u_{111},u_{121},u_{211},u_{221})$ as the independent variables.
Then,
\begin{align*}
  &u_{112}=u_{11\cdot}-u_{111}, \quad
   u_{122}=u_{12\cdot}-u_{121}, \quad
   u_{212}=u_{21\cdot}-u_{211}, \quad
   u_{222}=u_{22\cdot}-u_{221}, \\
  &u_{131}=u_{1\cdot1}-u_{111}-u_{121}, \quad
   u_{231}=u_{2\cdot1}-u_{211}-u_{221}, \quad
   u_{311}=u_{\cdot11}-u_{111}-u_{211}, \\
  &u_{321}=u_{\cdot21}-u_{121}-u_{221},\quad
   u_{331}=u_{3\cdot1}-u_{\cdot11}-u_{\cdot21}
   +u_{111}+u_{121}+u_{211}+u_{221},\\
  &u_{132}=u_{13\cdot}-u_{1\cdot1}+u_{111}+u_{121},\quad
   u_{232}=u_{23\cdot}-u_{2\cdot1}+u_{211}+u_{221},\\
  &u_{312}=u_{31\cdot}-u_{\cdot11}+u_{111}+u_{211},\quad
   u_{322}=u_{32\cdot}-u_{\cdot21}+u_{121}+u_{221},\\  
  &u_{332}=u_{33\cdot}-u_{3\cdot1}+u_{\cdot11}+u_{\cdot21}
   -u_{111}-u_{121}-u_{211}-u_{221}.
\end{align*}
The expression
\begin{align*}
y^u\propto&
\left(\frac{y_{111}y_{331}y_{132}y_{312}}{y_{131}y_{311}y_{112}y_{332}}\right)^{u_{111}}
\left(\frac{y_{121}y_{331}y_{132}y_{322}}{y_{131}y_{321}y_{122}y_{332}}\right)^{u_{121}}
\left(\frac{y_{211}y_{331}y_{232}y_{312}}{y_{231}y_{311}y_{212}y_{332}}\right)^{u_{211}}\\
&\times
\left(\frac{y_{221}y_{331}y_{232}y_{322}}{y_{231}y_{321}y_{222}y_{332}}\right)^{u_{221}}
=:z_{111}^{u_{111}}z_{121}^{u_{121}}z_{211}^{u_{211}}z_{221}^{u_{221}}
\end{align*}
implies that a generator of the toric ideal of $A^{(332)}$ contains
the following four binomials
\begin{align*}
&\partial_{111}\partial_{331}\partial_{132}\partial_{312}-
 \partial_{131}\partial_{311}\partial_{112}\partial_{332},\quad
 \partial_{121}\partial_{331}\partial_{132}\partial_{322}-
 \partial_{131}\partial_{321}\partial_{122}\partial_{332},\\
&\partial_{211}\partial_{331}\partial_{232}\partial_{312}-
 \partial_{231}\partial_{311}\partial_{212}\partial_{332},\quad
 \partial_{221}\partial_{331}\partial_{232}\partial_{322}-
 \partial_{231}\partial_{321}\partial_{222}\partial_{332},
\end{align*}
where $\partial_{ijk}=\partial/\partial y_{ijk}$. Diaconis and
Sturmfels \cite{DS98} called moves represented by quartic binomials
basic $2\times2\times2$ moves.
By the permutation of the indices of the above four binomials,
we have nine basic $2\times2\times2$ moves. In addition, we have
six moves of degree six, which appear in $y^u$ if we choose other
sets of independent variables. In total, there are 15 moves up to
sign and they form a minimal Markov basis. The normalized volume of
$A^{(332)}$, $\text{vol}(A^{(332)})$, was computed with the aid of the software {\tt polymake}
\cite{GJ00}, and the result was 81. Thus, the holonomic rank of
the $A$-hypergeometric ideal is 81. It is straightforward to see that
\begin{align*}
  &Z_{A^{(332)}}(b^{(332)};y)\propto\\&
  \sum_{u\in\mathcal{F}_{A^{(332)}}(b^{(322)})}
  \frac{
  \prod_{i\in[2],j\in[2]}
  (-u_{ij\cdot})_{u_{ij\cdot}-u_{ij2}}    
  \prod_{j\in[2]}(-u_{\cdot j1})_{u_{\cdot j1}-u_{3j1}}    
  \prod_{i\in[2]}(-u_{i\cdot1})_{u_{i\cdot 1}-u_{i31}}}
  {\prod_{i\in[2]}(-u_{i\cdot\cdot}+u_{i\cdot2}+u_{i3\cdot}+1)_{\sum_{j\in[2]}u_{ij1}}}\\
  &\times
  \frac{(u_{\cdot\cdot\cdot}+u_{3\cdot\cdot}+u_{\cdot3\cdot}+u_{\cdot\cdot2}-u_{33\cdot}-u_{3\cdot2}-u_{\cdot32})_{\sum_{i\in[2],j\in[2]}u_{ij1}}}
  {\prod_{j\in[2]}(-u_{\cdot j\cdot}+u_{\cdot j2}+u_{3j\cdot}+1)_{\sum_{i\in[2]}u_{ij1}}
  (-u_{\cdot\cdot1}+u_{3\cdot1}+u_{\cdot31}+1)_{\sum_{i\in[2],j\in[2]}u_{ij1}}}\\
  &\times\frac{z_{111}^{u_{111}}z_{121}^{u_{121}}z_{211}^{u_{211}}z_{221}^{u_{221}}}{u_{111}!u_{121}!u_{211}!u_{221}!},
\end{align*}
but an implementation of Algorithm~\ref{algo:1} seems prohibitive.

\section{Discussion}

This paper describes an interplay between the direct sampling
algorithm and the theory of hypergeometric functions of several
variables. The direct sampling algorithm proposes interesting
classes of hypergeometric functions and also the graphical
toric model gives an interesting new formula of hypergeometric
polynomials. On the opposite side, studies of contiguity relations
of hypergeometric functions will give an efficient direct sampler.
There are general algorithms to derive the relations, but they are
not efficient. New algorithms and theories for hypergeometric
functions are expected.

After this research was completed, one of the authors developed
an algorithm that, although approximate, achieves a significant
speedup of Algorithm~\ref{algo:0} compared to Algorithm~\ref{algo:1}
discussed in this paper \cite{Man25}. In addition, it was shown
that the approximate algorithm has reasonable accuracy and runs
even for the no-three-way interaction model described in
Example~\ref{exam:3way}. The contents explained in Section~\ref{sect:sum}
in this paper are shown to be properties of rational models
discussed by Duarte et al.~\cite{DMS21}.

\section*{Acknowledgements}

The first author was supported in part by JSPS KAKENHI Grants
18H00835 and 20K03742.
The second author was supported in part by the JST, CREST
Grant Number JP19209317.

\begin{flushleft}

Shuhei Mano\\
The Institute of Statistical Mathematics, Tokyo 190-8562, Japan\\
E-mail: smano@ism.ac.jp

\smallskip

Nobuki Takayama\\
Department of Mathematics, Graduate School of Science,
Kobe University, Kobe 657-8501, Japan\\
E-mail: takayama@math.kobe-u.ac.jp

\end{flushleft}

\section{Appendix}

\subsection{Derivation of \eqref{GHTdec}}

In this subsection we derive \eqref{GHTdec} by
the $A$-hypergeometric integral \eqref{GKZint}.
We make further preparations regarding the properties
of graphical models.

Let $B_1$, $\ldots$, $B_k$ be a sequence of the vertex set $V$
of an undirected graph $\mathcal{G}$. Let
\[
  H_j=B_1\cup\cdots\cup B_j,
  \quad S_{j+1}=H_j\cap B_{j+1}, \quad j \in [k-1]. 
\]
The sequence $B_1$, $\ldots$, $B_k$ is said to be {\it perfect}
if the following conditions are satisfied.
\begin{itemize}
  \item[(i)] for all $i>1$ there exists a $j<i$ such that
    $S_{i}\subseteq B_j$;
  \item[(ii)] the sets (separators) $S_i$ are complete for
    all $i$.
\end{itemize}
A separator $S$ may occur several times in a perfect sequence.
The number of occurrence $\nu(S)$ is called {\it multiplicity}.
In addition, a separator can be the empty set.

The {\it boundary} $\text{bd}(A)$ of a subset $A$ of vertices
is the set of vertices in $V\setminus A$ that are neighbours
to vertices in $A$, and the {\it closure} of a subset $A$ is
$\text{cl}(A)=A\cup\text{bd}(A)$. A {\it perfect numbering}
of the vertices $V$ of $\mathcal{G}$ is a numbering
$\alpha_1,\alpha_2,\ldots,\alpha_{|V|}$ such that
\[
  B_l=\text{cl}(\alpha_l)\cap\{\alpha_1,\ldots,\alpha_l\},
  \quad l\in [|V|]
\]
is a perfect sequence of sets. A characterization of chordal
graphs is as follows \cite[Proposition 2.17]{Lau94}.

\begin{proposition}\label{prop:perfect}
  The following conditions are equivalent for an undirected
  graph $\mathcal{G}$.
  \begin{itemize}
  \item[{\rm (i)}] the graph $\mathcal{G}$ is chordal.
  \item[{\rm (ii)}] the vertices of $\mathcal{G}$ admit
    a perfect numbering.
  \item[{\rm (ii)}] the cliques of $\mathcal{G}$ can be
    numbered to form a perfect sequence.
  \end{itemize}
\end{proposition}

If the boundary of a vertex is complete, then the vertex is
called {\it simplicial.} A vertex without neighbor is also
called simplicial. Since a clique consists of separators and
simplicial vertices, removing all the simplicial vertices
from a chordal graph induces backward elimination of cliques
from a perfect sequence.

\begin{example}
  Consider the simplicial complex
  $\Gamma=[123][124][135][246][247]$.
  The interaction graph $\mathcal{G}(\Gamma)$ is shown in
  the left of Figure~\ref{fig2}.
  \begin{figure}
  \includegraphics[width=\textwidth]{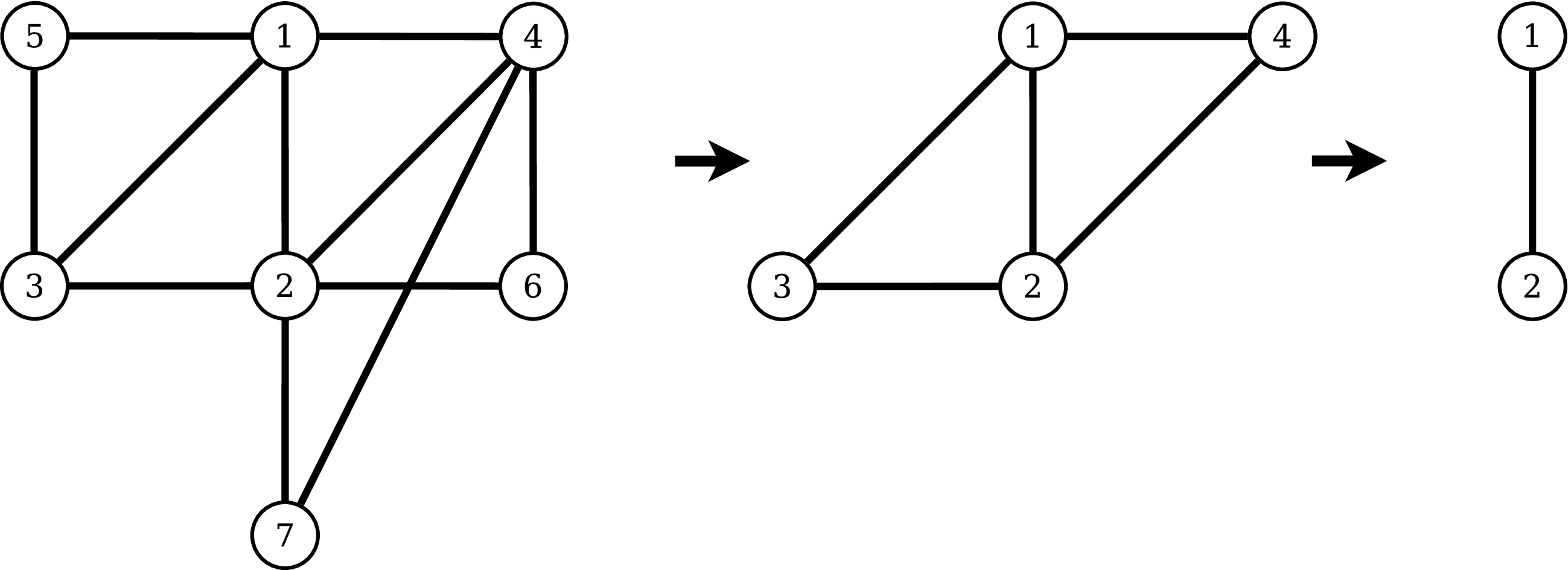}
  \caption{The process of removing simplicial vertices from
    the simplicial complex $[123][124][135][246][247]$.} 
  \label{fig2}
  \end{figure}
  The model is graphical since $\Gamma$ exactly consists
  of the cliques of $\mathcal{G}(\Gamma)$. Moreover,
  $\mathcal{G}(\Gamma)$ is chordal. The vertices of 
  $\mathcal{G}(\Gamma)$ are in a perfect numbering.
  A perfect sequence of cliques is $C_1=\{1,2,3\}$,
  $C_2=\{1,2,4\}$, $C_3=\{1,3,5\}$, $C_4=\{2,4,6\}$, and
  $C_5=\{2,4,7\}$. The separator for this perfect sequence
  is $S_2=\{1,2\},S_3=\{1,3\},S_4=S_5=\{2,4\}$, where
  the separator $\{2,4\}$ has the multiplicity of two.
  The vertices $7$, $6$, and $5$ are simplicial. Removing
  these simplicial vertices eliminates cliques $C_5$,
  $C_4$, and $C_3$, respectively. Then, we have
  the simplicial complex $[123][124]$ with the simplicial
  vertices $4$ and $3$. Removing these simplicial vertices
  eliminates cliques $C_2$ and $C_1$, respectively.
  The process of removing simplicial vertices from $\Gamma$
  is shown in Figure~\ref{fig2}.
\end{example}

As we have seen in Section~\ref{sect:sum}, the configuration
matrix $A$ of a simplicial complex $\Gamma=[F_1][F_2]\cdots$
determines the linear transformation
$u(i_V)\mapsto (u(i_{F_1}),u(i_{F_2}),\ldots)$, $i_V\in\mathcal{I}_V$,
$i_{F_1}\in\mathcal{I}_{F_1}$, $i_{F_2}\in\mathcal{I}_{F_2}$,
$\ldots$ Therefore, the rows of $A$ consists of the states of
the subsets in $\mathcal{I}_{F_1}$, $\mathcal{I}_{F_2}$, $\ldots$,
and the columns of $A$ consists of the states in $\mathcal{I}_V$.
The order of $i_F$ in the rows will be denoted by
${\rm Pos}(i_F)$ and that of $i_V$ in the columns will be
denoted by ${\rm Pos}(i_V)$. For example, for the simplicial
complex $\Gamma=[123][124]$,
${\rm Pos}(1111)=1$, ${\rm Pos}(1112)=2$, $\ldots$, and
${\rm Pos}(111\cdot)=1$, ${\rm Pos}(112\cdot)=2$, $\ldots$,
for the facet $F=[123]$. An expression of $A$ is given by
the following proposition. We can read off the correspondence
between the indices of the hierarchical log-linear model and
those of the configuration matrix. The derivation uses neither
graphical nor chordal. For a graphical model, the rows are
states of the cliques.

\begin{proposition}\label{prop:A}
  Consider a hierarchical log-linear model $\mathcal{M}_\Gamma$
  associated with the simplicial complex $\Gamma=[F_1][F_2]\cdots$.
  Let the rows and the columns of the configuration matrix $A$
  are states of the minimal sufficient statistics and
  the joint states, respectively, and they are ordered in
  the lexicographic order with $1\succ 2\succ\cdots$.
  For a facet $F\subset [m]$, define
  \begin{align*}
    &s_i:=\min\{j: j\in F, j>s_{i-1}, j-1\notin F\},\\
    &t_i:=\min\{j: j\notin F, j>t_{i-1}, j-1\in F\},
  \end{align*}
  for $1\le i\le \{j:t_j=m~\text{or}~m+1\}$ with $s_0=t_0=0$.
  Let
  \[
  A_F=1_{r_1\times \cdots \times r_{s_1-1}}\otimes
      E_{r_{s_1}\times \cdots \times r_{t_1-1}}\otimes
      1_{r_{t_1}\times \cdots \times r_{s_2-1}}\otimes
      E_{r_{s_2}\times \cdots \times r_{t_2-1}}\otimes\cdots,
  \]
  if $s_1>2$ and
  \[
  A_F=E_{r_{1}\times \cdots \times r_{t_1-1}}\otimes
      1_{r_{t_1}\times \cdots \times r_{s_2-1}}\otimes
      E_{r_{s_2}\times \cdots \times r_{t_2-1}}\otimes
      1_{r_{t_2}\times \cdots \times r_{s_3-1}}\otimes\cdots,
  \]
  if $s_1=1$. In other words, put $1$ for the maximal contiguous
  vertices which do not appear in $F$, and put $E$ for
  the maximal contiguous vertices in $F$. Then, the matrix $A$
  has the form
  \begin{equation}
  A_\Gamma=\left(
    \begin{array}{c}
    A_{F_1}\\
    A_{F_2}\\
    \vdots
    \end{array}  
    \right),
  \label{Ahie}
  \end{equation}
  where the column vector of the matrix $A_\Gamma$ is indexed
  by the elements of $\mathcal{I}_V$.
  Let $a_{ij}$ be the $(i,j)$-entry of the matrix $A_\Gamma$.
  For ${\rm Pos}(i_F)=i$ and ${\rm Pos}(i_V)=j$,
  $a_{ij} = 1$ if $i_V\subset i_F$ and $a_{ij}=0$ otherwise.
\end{proposition}

\begin{proof}
  To avoid messy expressions, the case of the facet $F=[13]$
  is shown. Similar proof works for any facet. If the $i$-th
  row of $A_F$ determines the minimal sufficient statistic
  $u_{p_1+1,\cdot,p_3}$, which is an element of $u(i_{[13]})$,
  $i=p_1 r_3+p_3$. For the $i$-th row, the $(i,j)$-entry of $A_F$
  is one if $j=p_1 (r_2r_3)+q_2 r_3+p_2$, $q_2\in 0\cup[r_2-1]$,
  and zero otherwise, that is,
  \[
    (a_{[13]})_{ij}
    =\sum_{q_2=0}^{r_2-1}\delta_{p_1 (r_2r_3)+q_2 r_3+p_2,j}.
  \]
  But the right hand side is equivalent to the $(i,j)$-entry
  of $E_{r_1}\otimes 1_{r_2} \otimes E_{r_3}$.
\end{proof}

We prepare the following proposition. Note that $t$ should be
indeterminants, because in the derivation of \eqref{GHTdec}
we use $t$ as variables for the integration along with
$\gamma\in{\mathbb{C}^{*d}}$. 

\begin{proposition}\label{prop:zero}
  If a polynomial $f(t) \in \mathbf{Q}[t_1,\ldots,t_d]$ is
  zero for all $t=T \in \mathbb{N}^d$,
  then we have $f(t)=0$ as a polynomial.
\end{proposition}

\begin{proof}
  We use an induction in $d$. When $d=1$, if the univariate
  polynomial $f(t)$ is not the zero polynomial, the number of
  zeros of $f(t)$ is up to the degree of $f(t)$. It
  contradicts the assumption, and $f(t)$ should be the zero
  polynomial. When $d=n>1$, the expansion in $t_n$ for some
  $m$:
  \[
    f(t)=\sum_{k=0}^m f_k(t') t_n^k, \quad t'=(t_1,\ldots,t_{n-1})
  \]
  determines polynomials
  $f_k(t')\in \mathbf{Q}[t_1, \ldots, t_{n-1}]$. If we fix
  $T'\in\mathbb{N}^{n-1}$, $f(t)$ reduces to a univariate
  polynomial, and the argument for $d=1$ gives $f_k(T')=0$.
  By the assumption of the induction for $d=n-1$, $f_k(t')=0$
  as a polynomial. Hence $f(t)$ is the zero polynomial.
\end{proof}

The following lemma for a factorization is the key for
the derivation of \eqref{GHTdec}.

\begin{lemma}\label{lemm:factor}
  Consider a decomposable graphical model associated with
  the simplicial complex $\Gamma$. We fix the state of
  all the separators $\mathcal{S}$ to be $i_{\mathcal{S}}$.
  For the set of indeterminants $t(i_V)$, $i_V\in\mathcal{I}_V$,
  we have
  \begin{equation}
    \sum_{i_V{\subset} i_{\mathcal{S}}}\prod_{C\in\mathcal{C}}
    t(i_{C})1\{{i_V\subset i_C}\}=\prod_{C\in\mathcal{C}}
    \sum_{i_C{\subset} i_{\mathcal{S}}}
    t(i_{C})
    \label{factor}
  \end{equation}
  as a polynomial equality, where $\mathcal{C}$ is the set
  of cliques of the interaction graph $\mathcal{G}(\Gamma)=(V,E)$,
  $t(i_F)=\sum_{j\in\mathcal{I}_{V\setminus F}}t(i_F,j)$ for
  any subset $F\subset V$, and $1\{\cdot\}$ is the indicator
  function, i.e. $1\{\cdot\}=1$ if $\cdot$ is true and zero
  otherwise.
\end{lemma}

\begin{proof}
  Consider a probability measure of $X$ taking values in
  the joint state space $\mathcal{I}_V$. We fix the state of
  all the separators $\mathcal{S}$ to be $i_{\mathcal{S}}$
  with satisfying $\mathbf{P}(X_{\mathcal{S}}=i_{\mathcal{S}})>0$.
  We may write
  \[
  \mathbf{P}(X_V=i_V|X_{\mathcal{S}}=i_{\mathcal{S}})
  =\frac{t(i_V)}{t(i_{\mathcal{S}})},
  \quad \forall i_V {\subset} i_{\mathcal{S}}.
  \]
  On the other hand, for each clique $C$ we may write
  \[
    \mathbf{P}(X_{C\setminus\mathcal{S}
    =i_{C\setminus\mathcal{S}}}|X_{\mathcal{S}}=i_{\mathcal{S}})
    =\frac{t(i_{C})}{\sum_{{i'_C\subset} i_{\mathcal{S}}}t({i'_{C}})},
    \quad \forall i_C{\subset} i_{\mathcal{S}}.
  \]
  Since the global Markov property \eqref{markov} implies
  \[
  \mathbf{P}(X_V=i_V|X_{\mathcal{S}}=i_{\mathcal{S}})
  =\prod_{C\in\mathcal{C}}\mathbf{P}(X_{C\setminus\mathcal{S}}=i_{C\setminus\mathcal{S}}|X_{\mathcal{S}}=i_{\mathcal{S}}),
  \]
  we have
  \[
    \frac{t(i_V)}{t(i_{\mathcal{S}})}=\prod_{C\in\mathcal{C}}
    \frac{t(i_{C})}{\sum_{{i'_C\subset }i_{\mathcal{S}}}t({i'_{C}})},
    \quad \forall i_V{\subset} i_{\mathcal{S}},
  \]
  where $i_V\subset i_C\subset i_{\mathcal{S}}$,
  $\forall C\in\mathcal{C}$. Since the denominator of the right
  hand side does not depend on $i_{V\setminus\mathcal{S}}$
  and
  $\sum_{i_{V}{\subset} i_\mathcal{S}}t(i_V)=t(i_{\mathcal{S}})$,
  we have
  \[
    \sum_{i_V{\subset} i_{\mathcal{S}}}\prod_{C\in\mathcal{C}}
    t(i_{C})1\{{i_V\subset i_C}\}=\prod_{C\in\mathcal{C}}
    \sum_{i_C{\subset} i_{\mathcal{S}}}t(i_{C}).
  \]
  Since this argument holds for any probability measure of $X$,
  Proposition~\ref{prop:zero} guarantees that \eqref{factor}
  holds even if $t(i_V)$, $i_V\in\mathcal{I}_V$ are indeterminants.
\end{proof}

\begin{proof}[Derivation of \eqref{GHTdec}]
  Let the interaction graph $\mathcal{G}(\Gamma)$ be a chordal
  graph with at least four vertices. We use a backward
  inductive argument with respect to the number of vertices.
  Figure~\ref{fig2} shows the process of integration for
  the simplicial complex $[123][124][135][246][247]$.
  The set of simplicial vertices is denoted by
  $V\setminus\mathcal{S}=:V_*$. Then, $\mathcal{G}(\Gamma)$
  has at least two simplicial vertices (Lemma 2.9 of \cite{Lau94}).
  If states of all separators are fixed to be $i_{\mathcal{S}}$,
  \begin{equation}
    \sum_{\{j:j\in[m],i_j{\subset} i_{\mathcal{S}}\}}
    \prod_{i\in[d]}t_i^{a_{ij}}
    =\sum_{\{j:j\in[m],i_j{\subset} i_{\mathcal{S}}\}}
    \prod_{C\in\mathcal{C}}t_{i_{C}}1\{{i_j\subset i_C}\}
    =\prod_{C\in\mathcal{C}}
    \sum_{i_C{\subset}i_{\mathcal{S}}}t_{i_{C}},
    \label{sumdec}
  \end{equation}
  where the factorization of the second equality follows by
  the formula \eqref{factor} of Lemma~\ref{lemm:factor}
  and Proposition~\ref{prop:zero}. The multinomial expansion
  yields
  \begin{align*}
    \left(\sum_{j\in[m]}\prod_{i\in[d]}t_i^{a_{ij}}\right)^n
    &=\sum_{\{v(i_{\mathcal{S}}):\sum_{i_{\mathcal{S}}} v(i_{\mathcal{S}})=n\}}
    \frac{n!}{\prod_{i_{\mathcal{S}}} v(i_{\mathcal{S}})!}
    \prod_{i_{\mathcal{S}}}\prod_{C\in\mathcal{C}}
    \left(\sum_{i_C{\subset} i_{\mathcal{S}}}t_{i_C}\right)^{v(i_{\mathcal{S}})}\\
    &=\sum_{\{v(i_{\mathcal{S}}):\sum_{i_{\mathcal{S}}} v(i_{\mathcal{S}})=n\}}
    \frac{n!}{\prod_{i_{\mathcal{S}}} v(i_{\mathcal{S}})!}
    \prod_{C\in\mathcal{C}}
    \prod_{i_{\mathcal{S}}}\left(\sum_{i_C{\subset} i_{\mathcal{S}}}
    t_{i_C}\right)^{v(i_{\mathcal{S}})}.
  \end{align*}
  Here $v(i_{\mathcal{S}})$ is a variable which takes a value in
  $\mathbb{N}_0$ indexed by the elements of $\mathcal{I}_{\mathcal{S}}$.
  If a clique does not contain simplicial vertices, the indices
  of the clique are completely specified by the condition
  $i_C{\subset} i_{\mathcal{S}}$ and we have
  \[
    \prod_{i_{\mathcal{S}}}\left(\sum_{i_C{\subset} i_{\mathcal{S}}}
    t_{i_C}\right)^{v(i_{\mathcal{S}})}=\prod_{i_C}t_{i_C}^{v(i_C)},
  \]
  Otherwise,
  \begin{align*}
    &\prod_{i_{\mathcal{S}}}\left(\sum_{i_C{\subset} i_{\mathcal{S}}}
    t_{i_C}\right)^{v(i_{\mathcal{S}})}
    =\prod_{i_{C\setminus V_*}}
    \left(\sum_{i_C{\subset} i_{C\setminus V_*}}
    t_{i_C}\right)^{v(i_{C\setminus V_*})}\\
    &=\prod_{i_{C\setminus V_*}}
    \sum_{\{v(i_C):\sum_{i_C{\subset} i_{C\setminus V_*}}
    v(i_C)=v(i_{C\setminus V_*})\}}
    \frac{v(i_{C\setminus V_*})!}
    {\prod_{i_C{\subset} i_{C\setminus V_*}}v(i_C)!}
    \prod_{i_C{\subset} i_{C\setminus V_*}}t_{i_C}^{v(i_C)}.
  \end{align*}
  Note that $C\setminus V_*\subset S$ and the inclusion
  can be strict. Finally,
  \[
    \prod_{i\in[d-1]}t_i^{-b_i-1}dt_i=
    \prod_{C\in\mathcal{C}}\prod_{i_C}t_{i_C}^{-u(i_C)-1}
    dt_{i_C}.
  \]
  The integrand is holomorphic on $\mathbb{C}^{*d-1}$.
  By using Cauchy's integral theorem, the integral \eqref{GKZint}
  with taking the $(d-1)$-cycle $\gamma$ as the $(d-1)$-complex torus
  $T^{d-1}=\mathbb{C}^{*d-1}$
  around the origin is evaluated as
  \begin{align*}
    &Z_{A_\Gamma}(b_\Gamma;1)\propto
    \sum_{\{v(i_{\mathcal{S}}):\sum_{i_{\mathcal{S}}}v(i_{\mathcal{S}})=n\}}
    \frac{\prod_{C\in\mathcal{C}_0}\prod_{i_C}1(v(i_C)=u(i_C))}
         {\prod_{i_{\mathcal{S}}}v(i_{\mathcal{S}})!}\\
    &\times\prod_{C \in\mathcal{C}\setminus\mathcal{C}_0}
    \prod_{i_{C\setminus V_*}}
    \sum_{
      \{v(i_C):\sum_{i_C{\subset} i_{C\setminus V_*}}v(i_C)=v(i_{C\setminus V_*})\}}
    \frac{v(i_{C\setminus V_*})!}{\prod_{i_C{\subset} i_{C\setminus V_*}}v(i_C)!}1(v(i_C)=u(i_C))\\
    &=
    \sum_{\{v(i_{\mathcal{S}}):\sum_{i_{\mathcal{S}}} v(i_{\mathcal{S}})=n\}}
    \frac{\prod_{C\in\mathcal{C}_0}\prod_{i_C}1(v(i_C)=u(i_C))}{\prod_{i_{\mathcal{S}}}v(i_{\mathcal{S}})!}
    \prod_{C \in\mathcal{C}\setminus\mathcal{C}_0}
    \prod_{i_{C\setminus V_*}}\frac{u(i_{C\setminus V_*})!}
    {\prod_{i_C{\subset} i_{C\setminus V_*}}u(i_C)!}
  \end{align*}
  where $\mathcal{C}_0$ is the set of cliques containing no
  simplicial vertices. The summation appears in the last expression
  is the $A$-hypergeometric polynomial associated with the graph
  $\mathcal{G}(\Gamma)$ with removing $V_*$, say
  $Z_{A_{\Gamma_0}}(b_{\Gamma_0};1)$. A chordal graph whose
  simplicial vertices are removed is again a chordal graph. 
  By the assumption of the induction, we should have
  \[
    Z_{A_{\Gamma_0}}(b_{\Gamma_0};1)=
    \frac{\prod_{S\in\mathcal{S}}\{\prod_{i_S\in\mathcal{I}_S}
    u(i_S)!\}^{\nu_0(S)}}
    {\prod_{C\in\mathcal{C}_0}\prod_{i_C\in\mathcal{I}_C}u(i_C)!},
  \]
  where $\nu_0(S)$ is the multiplicity of separator $S$
  in a perfect sequence of the cliques in the set $\mathcal{C}_0$.
  This observation proves the assertion for
  $Z_{A_{\Gamma}}(b_{\Gamma};1)$, because when we remove
  $V_*$ from $\mathcal{G}(\Gamma)$ the set
  $C\setminus V_*$, $C\in\mathcal{C}\setminus\mathcal{C}_0$
  becomes a separator. Therefore, we have
  \[
    \prod_{C \in\mathcal{C}\setminus\mathcal{C}_0}
    \prod_{i_{C\setminus V_*}}u(i_{C\setminus V_*})!
    =\prod_{S\in\mathcal{S}}\{\prod_{i_S\in\mathcal{I}_S}
    u(i_S)!\}^{\nu(S)-\nu_0(S)}.
  \]
  By the backward induction, our task reduces to confirm
  the assertion for a set of disconnected simplices. If
  the set consists of a single simplex, the computation is
  trivial, because $A=E_m$ and $Z_A(b;1)=m^n/n!$. Otherwise,
  an $A$-hypergeometric integral is involved. We discuss
  the case with two vertices, but other cases can be discussed
  in similar ways. The graph coincides that of the independence
  model of two-way contingency tables discussed in
  Section~\ref{subs:2}. For simplicity, we display explicit
  expressions for the binary case. Taking the 3-cycle $\gamma$
  as the 3-complex torus $T^3$ around the origin, the integral
  becomes
  \begin{align*}
    &\int_{T^3}
    (t_{1\cdot}+t_{2\cdot})^n(1+t_{\cdot 1})^n
    t_{1\cdot}^{-u_{1\cdot}-1}t_{2\cdot}^{-u_{2\cdot}-1}
    t_{\cdot1}^{-u_{\cdot1}-1}dt_{1\cdot}dt_{2\cdot}
    dt_{\cdot1}\\
    &=\int_{T^3}
    \sum_{v_{1\cdot}=0}^n\sum_{v_{\cdot1}=0}^n
    \left(\begin{array}{c}n\\v_{1\cdot}\end{array}\right)
    \left(\begin{array}{c}n\\v_{\cdot1}\end{array}\right)
    t_{1\cdot}^{v_{1\cdot}-u_{1\cdot}-1}
    t_{2\cdot}^{n-v_{1\cdot}-u_{2\cdot}-1}
    t_{\cdot1}^{v_{\cdot1}-u_{\cdot1}-1}dt_{1\cdot}dt_{2\cdot}
    dt_{\cdot1}\\
    &=
    \int_{T}\sum_{v_{\cdot1}=0}^n
    \left(\begin{array}{c}n\\v_{\cdot1}\end{array}\right)
    t_{\cdot1}^{v_{\cdot1}-u_{\cdot1}-1}dt_{\cdot1}
    \int_{T^2}\sum_{v_{1\cdot}=0}^n
    \left(\begin{array}{c}n\\v_{1\cdot}\end{array}\right)
    t_{1\cdot}^{v_{1\cdot}-u_{1\cdot}-1}
    t_{2\cdot}^{n-v_{1\cdot}-u_{2\cdot}-1}dt_{1\cdot}dt_{2\cdot}\\
    &=2\pi\sqrt{-1}\left(\begin{array}{c}n\\u_{\cdot1}\end{array}
    \right)
    \int_{T}\left\{\int_{T}
    \sum_{v_{1\cdot}=0}^n
    \left(\begin{array}{c}n\\v_{1\cdot}\end{array}\right)
    t_{1\cdot}^{v_{1\cdot}-u_{1\cdot}-1}
    t_{2\cdot}^{u-v_{1\cdot}-u_{2\cdot}-1}dt_{1\cdot}\right\}
    dt_{2\cdot}\\
    &=(2\pi\sqrt{-1})^2
    \left(\begin{array}{c}n\\u_{\cdot1}\end{array}\right)
    \left(\begin{array}{c}n\\u_{1\cdot}\end{array}\right)
    \int_{T}t_{2\cdot}^{-1}dt_{2\cdot}
    =\frac{(2\pi\sqrt{-1})^3n!^2}
    {u_{1\cdot}!u_{2\cdot}!u_{\cdot1}!u_{\cdot2}!}.
  \end{align*}
  where the last three equalities follow by Cauchy's integral
  theorem. The proportional constant can be fixed by the case
  that $b$ is the zero vector, and we obtain
  the expression \eqref{GHTdec} for the $A$-hypergeometric polynomial.
\end{proof}  

\begin{example}\label{exam:graph6}
  Consider a binary decomposable graphical model associated with the simplicial complex
  $\Gamma=[123][124][135][246]$.
  The simplicial vertices are $5$ and $6$. A perfect
  sequence of cliques is $C_1=\{1,2,3\}$,
  $C_2=\{1,2,4\}$, $C_3=\{1,3,5\}$ and $C_4=\{2,4,6\}$, and
  the separators are $S_2=\{1,2\}$, $S_3=\{1,3\}$, and
  $S_4=\{2,4\}$. The summation \eqref{sumdec} in the derivation of
  \eqref{GHTdec} consists of terms with
  fixing possible set of indices of $\mathcal{S}=\{1,2,3,4\}$
  becomes
  \begin{align*}
    \sum_{j\in[2^6]}\prod_{i\in[8\times4]}t_i^{a_{ij}}
    =&t_{111\cdot\cdot\cdot}t_{11\cdot1\cdot\cdot}
    (t_{1\cdot1\cdot1\cdot}t_{\cdot1\cdot1\cdot1}
    +t_{1\cdot1\cdot1\cdot}t_{\cdot1\cdot1\cdot2}
    +t_{1\cdot1\cdot2\cdot}t_{\cdot1\cdot1\cdot1}
    +t_{1\cdot1\cdot2\cdot}t_{\cdot1\cdot1\cdot2})\\
    &+\cdots\\
    =&t_{111\cdot\cdot\cdot}t_{11\cdot1\cdot\cdot}
    (t_{1\cdot1\cdot1\cdot}+t_{1\cdot1\cdot2\cdot})
    (t_{\cdot1\cdot1\cdot1}+t_{\cdot1\cdot1\cdot2})
    +(\text{other~15~terms}),
  \end{align*}
  where the first term of the right hand side appears for
  $i_{\mathcal{S}}=(1111\cdot\cdot)$.
  The second equality comes from the factorization formula
  \eqref{factor} in Lemma~\ref{lemm:factor} with the global
  Markov property $X_5\perp\!\!\!\perp X_6|X_{\{1,2,3,4\}}$,
  where the states of the cliques $C_1$ and $C_2$ are uniquely
  determined by $i_{\mathcal{S}}$. The multinomial expansion
  of the $n$-th power of this expression is
  \begin{align*}
    &\sum
    \frac{n!}{v_{1111\cdot\cdot}!\cdots v_{2222\cdot\cdot}!}
    \{t_{111\cdot\cdot\cdot}t_{11\cdot1\cdot\cdot}
    (t_{1\cdot1\cdot1\cdot}+t_{1\cdot1\cdot2\cdot})
    (t_{\cdot1\cdot1\cdot1}+t_{\cdot1\cdot1\cdot2})
    \}^{v_{1111\cdot\cdot}}\cdots\\
    =&\sum
    \frac{n!}{v_{1111\cdot\cdot}!\cdots v_{2222\cdot\cdot}!}
    t_{111\cdot\cdot\cdot}^{v_{111\cdot\cdot\cdot}}
    t_{11\cdot1\cdot\cdot}^{v_{11\cdot1\cdot\cdot}}
    (t_{1\cdot1\cdot1\cdot}+t_{1\cdot1\cdot2\cdot})^{v_{1\cdot1\cdot\cdot\cdot}}
    (t_{\cdot1\cdot1\cdot1}+t_{\cdot1\cdot1\cdot2})^{v_{\cdot1\cdot1\cdot\cdot}}\cdots\\
    =&\sum
    \frac{n!}{v_{1111\cdot\cdot}!\cdots v_{2222\cdot\cdot}!}
    \frac{v_{1\cdot1\cdot\cdot\cdot}!v_{\cdot1\cdot1\cdot\cdot}!
          v_{1\cdot2\cdot\cdot\cdot}!v_{\cdot1\cdot2\cdot\cdot}!
          v_{2\cdot1\cdot\cdot\cdot}!v_{\cdot2\cdot1\cdot\cdot}!
          v_{2\cdot2\cdot\cdot\cdot}!v_{\cdot2\cdot2\cdot\cdot}!}
        {v_{1\cdot1\cdot1\cdot}!v_{1\cdot1\cdot2\cdot}!
         v_{\cdot1\cdot1\cdot1}!v_{\cdot1\cdot1\cdot2}!\cdots
         v_{2\cdot2\cdot1\cdot}!v_{2\cdot2\cdot2\cdot}!
         v_{\cdot2\cdot2\cdot1}!v_{\cdot2\cdot2\cdot2}!}\\
    &\times
        t_{111\cdot\cdot\cdot}^{v_{111\cdot\cdot\cdot}}
        t_{11\cdot1\cdot\cdot}^{v_{11\cdot1\cdot\cdot}}\cdots
        t_{222\cdot\cdot\cdot}^{v_{222\cdot\cdot\cdot}}
        t_{22\cdot2\cdot\cdot}^{v_{22\cdot2\cdot\cdot}}
        t_{1\cdot1\cdot1\cdot}^{v_{1\cdot1\cdot1\cdot}}
        t_{1\cdot1\cdot2\cdot}^{v_{1\cdot1\cdot2\cdot}}
        t_{\cdot1\cdot1\cdot1}^{v_{\cdot1\cdot1\cdot1}}
        t_{\cdot1\cdot1\cdot2}^{v_{\cdot1\cdot1\cdot2}}\cdots
        t_{2\cdot2\cdot1\cdot}^{v_{2\cdot2\cdot1\cdot}}
        t_{2\cdot2\cdot2\cdot}^{v_{2\cdot2\cdot2\cdot}}
        t_{\cdot2\cdot2\cdot1}^{v_{\cdot2\cdot2\cdot1}}
        t_{\cdot2\cdot2\cdot2}^{v_{\cdot2\cdot2\cdot2}}.
  \end{align*}
  The $A$-hypergeometric integral leads to
  \begin{align*}
    &\frac{u_{1\cdot1\cdot\cdot\cdot}!u_{\cdot1\cdot1\cdot\cdot}!
           u_{1\cdot2\cdot\cdot\cdot}!u_{\cdot1\cdot2\cdot\cdot}!
           u_{2\cdot1\cdot\cdot\cdot}!u_{\cdot2\cdot1\cdot\cdot}!
           u_{2\cdot2\cdot\cdot\cdot}!u_{\cdot2\cdot2\cdot\cdot}!}
          {u_{1\cdot1\cdot1\cdot}!u_{1\cdot1\cdot2\cdot}!
           u_{\cdot1\cdot1\cdot1}!u_{\cdot1\cdot1\cdot2}!\cdots
           u_{2\cdot2\cdot1\cdot}!u_{2\cdot2\cdot2\cdot}!
           u_{\cdot2\cdot2\cdot1}!u_{\cdot2\cdot2\cdot2}!}\\
       &\times\sum_{\{(v_{1111\cdot\cdot},\ldots,v_{2222\cdot\cdot}):
       v_{111\cdot\cdot\cdot}=u_{111\cdot\cdot\cdot},\ldots,
       v_{22\cdot2\cdot\cdot}=u_{22\cdot2\cdot\cdot}\}}
       \frac{1}{v_{1111\cdot\cdot}!\cdots v_{2222\cdot\cdot}!}
  \end{align*}
  up to a constant. Here, the summation is the $A$-hypergeometric
  polynomial associated with the simplicial complex $[123][124]$,
  which was already obtained in Example~\ref{exam:4cyclet}.
  The expression
  \[
  \frac{u_{11\cdot\cdot\cdot\cdot}!u_{12\cdot\cdot\cdot\cdot}!
        u_{21\cdot\cdot\cdot\cdot}!u_{22\cdot\cdot\cdot\cdot}!}
       {u_{111\cdot\cdot\cdot}!\cdots u_{22\cdot2\cdot\cdot}!},
  \]
  confirms the formula \eqref{GHTdec}.     
\end{example}  

\subsection{The Pfaffian system for the binary no-three-way interaction model}

The Pfaffian system is given by the following matrices
(the fourth rows are specified later):
\[
\tilde{P}_8=
\left(
\begin{array}{cccc}
  0&1&0&0\\0&0&1&0\\0&0&0&1\\ *&*&*&*
\end{array}
\right), \quad
\tilde{P}_1=
\left(
\begin{array}{cccc}
  -c_4&-1&0&0\\0&-c_4&-1&0\\0&0&-c_4&-1\\ *&*&*&*
\end{array}
\right),
\]
\[
\tilde{P}_2=
\left(
\begin{array}{cccc}
  c_4-c_1&1&0&0\\0&c_4-c_1&1&0\\0&0&c_4-c_1&1\\ *&*&*&*
\end{array} 
\right),
\]
\[
\tilde{P}_4=
\left(
\begin{array}{cccc}
  c_4-c_2&1&0&0\\0&c_4-c_2&1&0\\0&0&c_4-c_2&1\\ *&*&*&*
\end{array} 
\right),
\]
\[
\tilde{P}_6=
\left(
\begin{array}{cccc}
  c_4-c_3&1&0&0\\0&c_4-c_3&1&0\\0&0&c_4-c_3&1\\ *&*&*&*
\end{array}
\right),
\]
\[
\tilde{P}_3=
\left(
\begin{array}{cccc}
  d_2-c_4-1&-1&0&0\\0&d_2-c_4-1&-1&0\\0&0&d_2-c_4-1&1\\ *&*&*&*
\end{array}
\right),
\]
\[
\tilde{P}_5=
\left(
\begin{array}{cccc}
  d_3-c_4-1&-1&0&0\\0&d_3-c_4-1&-1&0\\0&0&d_3-c_4-1&1\\ *&*&*&*
\end{array}
\right),
\]
\[
\tilde{P}_7=
\left(
\begin{array}{cccc}
  d_4-c_4-1&-1&0&0\\0&d_4-c_4-1&-1&0\\0&0&d_4-c_4-1&1\\ *&*&*&*
\end{array}
\right).
\]
The entries of the fourth rows of these matrices can be obtained
as follows. Since $\sum_{i=1}^4(d_i-c_i)-4=u_{\cdot\cdot\cdot}>0$,
we can use the recurrence relation \eqref{cl_rec1} denoted by
\[
  \theta^4=r_3(z)\theta^3+r_2(z)\theta^2+r_1(z)\theta+r_0(z).
\]
With using
\[
  \theta_8^iZ_{A^{(2^3)}}(b;z)
  =cy^s(-c_4-\theta)^i{}_4F_3(c;d;z), \quad i\in\mathbb{N}
\]
and
\[
  cy^s\theta^i{}_4F_3(c;d;z)=(-c_4-\theta_8)^iZ_{A^{(2^3)}}(b;z),
  \quad i\in\mathbb{N},
\]
and the recurrence relation \eqref{cl_rec2} denoted by
\[
  \theta^4=r_3\theta^3+r_2\theta^2+r_1\theta+r_0,
\]
where $\sum_{i=1}^4(d_i-c_i)-5=u_{\cdot\cdot\cdot}-1>0$,
we can read off the fourth row of $P_8$ from the expression
\begin{align*}
  \theta_8^4Z_{A^{(2^3)}}(b;y)=
  &\{(-4c_4-r_3)\theta_8^3+(-6c_4^2-3r_3c_4+r_2)\theta_8^2\\
  &+(-4c_4^3-3r_3c_4^2+2r_2c_4-r_1)\theta_8\\
  &+(-c_4^4-r_3c_4^3+r_2c_4^2-r_1c_4+r_0)\}Z_{A^{(2^3)}}(b;y)\\
  =&\{(p_8)_{41}+(p_8)_{42}\theta_8+(p_8)_{43}\theta_8^2
   +(p_8)_{44}\theta_8^3\}Z_{A^{(2^3)}}(b;y).
\end{align*}
The entries of the fourth row of the other matrices can be
obtained in a similar way.
  
\end{document}